\documentclass[11pt,letterpaper]{article}
\usepackage[top=0.85in,left=2.75in,footskip=0.75in,marginparwidth=2in]{geometry}

\usepackage[utf8]{inputenc}

\usepackage{cite}

\usepackage{nameref,hyperref}
\usepackage{xcolor}
\hypersetup{
   colorlinks=true, urlcolor=red, citecolor={red}, pdfborder={0 0 0}
}

\usepackage{microtype}
\DisableLigatures[f]{encoding = *, family = * }

\usepackage{changepage}

\usepackage[aboveskip=1pt,labelfont=bf,labelsep=period,singlelinecheck=off]{caption}


\usepackage{lastpage,fancyhdr,graphicx}
\usepackage{epstopdf}
\fancyheadoffset[L]{2.25in}
\fancyfootoffset[L]{2.25in}

\usepackage{color}

\definecolor{Gray}{gray}{.25}

\usepackage{graphicx}

\usepackage{sidecap}

\usepackage{wrapfig}
\usepackage[pscoord]{eso-pic}
\usepackage[fulladjust]{marginnote}
\reversemarginpar

\usepackage{amsthm}
\usepackage[normalem]{ulem} 
\newtheorem{theorem}{Theorem}[section]

\usepackage{nicematrix}

\theoremstyle{remark}
\newtheorem{remark}{Remark}[section]
\usepackage{fontawesome5}
\usepackage{bm}
\usepackage{amsfonts}

\usepackage{environ}
\usepackage{changepage} 
\NewEnviron{ADJequation}{
  \begin{adjustwidth}{-1.5in}{0in}
    \begin{equation}
        \BODY
    \end{equation}
\end{adjustwidth}
}

\usepackage{orcidlink}
\usepackage{marginnote}
\usepackage{lipsum}
\usepackage{multirow} 
\usepackage{enumitem}
\usepackage{enumerate}
\usepackage{ulem}
\begin{document}
\vspace*{0.35in}

\begin{flushleft}
{\Large\textbf{A Structure-Preserving Rational Integrator for the Replicator Dynamics on the Probability Simplex}
}
\newline
\\
{\Large Mario Pezzella\textsuperscript{1,2,*}\orcidlink{0000-0002-1869-945X}}
\\
\bigskip
1 National Research Council of Italy, Institute for Applied Mathematics \\ \phantom{1} “Mauro Picone”, Via P. Castellino, 111 - 80131, Naples, Italy
\\
2 Member of the INdAM Research group GNCS, Rome, Italy
\\
* Corresponding Author \hspace{0.277 \textwidth} \faEnvelope[regular] \href{mailto:mario.pezzella@cnr.it}{mario.pezzella@cnr.it}
\end{flushleft}

\medskip

\begin{abstract}
    In this work, we introduce a quadratically convergent and dynamically consistent integrator specifically designed for the replicator dynamics. The proposed scheme combines a two-stage rational approximation with a normalization step to ensure confinement to the probability simplex and unconditional preservation of non-negativity, invariant sets and equilibria. A rigorous convergence analysis is provided to establish the scheme’s second-order accuracy, and an embedded auxiliary method is devised for adaptive time-stepping based on local error estimation. Furthermore, a discrete analogue of the quotient rule, which governs the evolution of component ratios, is shown to hold. Numerical experiments validate the theoretical results, illustrating the method’s ability to reproduce complex dynamics and to outperform well-established solvers in particularly challenging scenarios.
\end{abstract}
{ \noindent \small  \textsc{\textbf{Keywords.}} Rational approximation, \ replicator system, \ structure-preserving scheme,   \ high order, \ dynamical consistency.}\\
{ \noindent \small  \textsc{\textbf{MSC2020.}} \ 65L05, 65L70, 65Z05, 37M15, 41A20}

\section{Introduction}
The replicator equation emerges in a wide spectrum of scientific disciplines, where it models the competitive dynamics among interacting species or strategies shaped by frequency-dependent selection. In its classical continuous-time formulation, the replicator system is given by
\begin{equation}\label{eq:replicator}
  \frac{dx_i}{dt}(t) = x_i(t)\left( f_i(\bm{x}(t)) - \sum_{j=1}^{N} x_j(t) f_j(\bm{x}(t)) \right),\quad i=1,\dots,N,
\end{equation}
where $\bm{x}(t)=[x_1(t),\dots,x_N(t)]^\mathsf{T}\in \mathbb{R}^N$ denotes the state vector and $f_i(\bm{x}(t))$, $i=1,\dots,N$, represents the fitness associated with the $i$-th component. The primary objective of this work is to construct a high-order, positivity-preserving numerical method for system~\eqref{eq:replicator}, which unconditionally retains key structural features of the continuous dynamics constrained to the probability simplex.

Originally introduced by Taylor and Jonker~\cite{Taylor1978} and later named by Schuster and Sigmund~\cite{SCHUSTER}, the replicator framework was formalized in a biological setting by Maynard Smith and Price~\cite{SMITH_1973,MAYNARDSMITH1974209} to describe natural selection independently of genetic mechanisms. It also laid the foundations of evolutionary game theory, where replicator equations describe how populations adjust their strategic behavior in response to payoff differences. In this setting, the equilibria of the system correspond to the Nash equilibria of the underlying game, offering a dynamical perspective on classical notions of game-theoretic stability~\cite{Sandholm2010,Random,Natalini1,Natalini2}.
Since then, it has become a cornerstone of theoretical biology, with applications ranging from adaptation theory to ecological and genetic modeling~\cite{Hofbauer1998,Nowak2006}. In theoretical ecology, the replicator equation can be derived from generalized Lotka–Volterra systems reformulated in one higher dimension~\cite{Hofbauer1998,Hofbauer_2003}, a connection that has motivated studies on polymorphic populations and phenotype distributions~\cite{Varga_2024}. In population genetics, it provides a deterministic approximation of allele frequency dynamics in diploid, randomly mating populations under selection~\cite{Hofbauer_2003,Ewens_2004}. More recently, replicator dynamics have gained renewed relevance in the modeling of pathogen competition and infectious disease transmission, particularly in scenarios involving multiple co-circulating strains~\cite{Gjini_2023,Madec_2020}.

Beyond biological contexts, equation~\eqref{eq:replicator} has found increasing application in computational sciences. It has been employed to model strategic behavior on complex networks~\cite{Li_2025,Networks2}, to identify dominant sets in clustering and pattern recognition~\cite{Clustering,Pavan2007} and to tackle problems in combinatorial optimization~\cite{Bomze_1997}, multi-agent reinforcement learning~\cite{Sato2002} and evolutionary finance~\cite{Weibull1995}.

Due to this diversity of applications, the efficient numerical integration of system~\eqref{eq:replicator} presents practical challenges. In particular, standard numerical schemes often fail to retain the structural properties and invariants of the continuous system, unless the time step is adequately refined. As a result, numerical simulations may lose long-term reliability and interpretability in applied contexts. Such considerations, combined with the lack of structure-preserving schemes specifically devised for the replicator system, motivate the present study and the development of a second-order, dynamically consistent integrator for model~\eqref{eq:replicator}. 

Several contributions in the literature have addressed similar goals in different settings. Particular emphasis has been placed on the numerical preservation of positivity, a fundamental requirement in many dynamical systems arising in real-world applications. In this regard, NonStandard Finite Difference (NSFD) schemes~\cite{Mickens_NSFD} have attracted considerable attention, with numerous extensions proposed for both differential~\cite{Lubuma_Chem,Conte2022,Patidar,Anguelov} and non-local models~\cite{MPV_MBE,MPV_JCD,Pezzella_ESAIM,Cardone_Fract,Fract}. Although such schemes are inherently designed to ensure dynamical consistency, their accuracy is often constrained by a low order of convergence~\cite{Hoang03102023,NSFD_High,Hoang_Second}. In general, the construction of high-order unconditionally positive numerical methods remains a challenging task, typically limited to specific classes of systems (see, e.g.,~\cite{PDS,MPV_Axioms,MPV_mixing,OFFNER202015,CdS}). More broadly, the simultaneous preservation of additional structural features, such as invariant sets, equilibrium configurations and asymptotic dynamics, has inspired the development of geometric and structure-preserving integrators~\cite{Hairer2006, Liu, LEWIS2003141, Sharma, GeCO_Izgin, GeC01, Celledoni2014, MPV_mixing, MPSL}.

The manuscript is organized as follows. Section~\ref{sec:Model} recalls classical results from the literature concerning the replicator model~\eqref{eq:replicator} and outlines the main structural properties of its dynamics, including the evolution on the probability simplex, the existence of internal equilibrium points and the so-called \emph{quotient rule}. The normalized rational integrator then is introduced in Section~\ref{sec:Method}, where we prove that it preserves non-negativity and the linear invariant of the system for arbitrary stepsize sequences, ensuring that the discrete dynamics remain confined to the simplex. A comprehensive error analysis is presented in Section~\ref{Sec:Error} and theoretical results on consistency, convergence and adherence to the \emph{quotient rule} are provided. An adaptive time-stepping strategy based on local error control is there proposed, as well. Section~\ref{sec:Numerical_Experiments} presents numerical experiments, while Section~\ref{sec:Conclusions} concludes the paper with remarks and insights on future developments.

\section{The continuous-time model}\label{sec:Model}
We consider the replicator model \eqref{eq:replicator} compactly reformulated as follows
\begin{equation}\label{eq:replicator_Compact}
  \dfrac{d\bm{x}}{dt}(t) = \bm{x}(t) \odot \left( \bm{f}(\bm{x}(t)) - \left( \bm{f}(\bm{x}(t))^\mathsf{T}  \bm{x}(t)\right)  \, \bm{e} \right)
\end{equation}
where $\bm{x}(t) = [x_1(t),\dots,x_N(t)]^\mathsf{T} \in \mathbb{R}^N$  is the non-negative state vector, $\bm{f}(\bm{x}) = [f_1(\bm{x}),\dots,f_N(\bm{x})]^\mathsf{T} \in \mathbb{R}^N$ is the vector of fitness functions, $\bm{e} = [1,\dots,1]^\mathsf{T}$ is the vector of all ones and $\odot$ denotes the Hadamard (component-wise) product. There is ample literature devoted to the study and analysis of models of the form~\eqref{eq:replicator_Compact}, with foundational contributions in \cite{Sigmund1986,SCHUSTER}. In what follows, we briefly recall the main structural features of the model, which will guide the construction of dynamics-preserving numerical schemes.
\begin{itemize}
    \item The $N-1$ dimensional probability simplex $\Delta^{N-1}\subset \mathbb{R}^N$ (see, for instance, \cite[Section 2.2.4]{boyd2004convex}), defined as the set 
    \begin{equation}\label{eq:simplex_def}
        \Delta^{N-1} = \left\{\bm{\xi}=[\xi_1,\dots,\xi_N]^\mathsf{T}\in\mathbb{R}^N \ :  \ \sum_{i=1}^{N}\xi_i = 1 \text{ and } \xi_i\ge 0, \ \forall \, i=1,\dots,N \right\},
    \end{equation}
    is positively invariant under the flow of the replicator dynamics, that is,\begin{equation}\label{eq:Simplex_Invariant}
        \bm{x}(0) \in \Delta^{N-1} \quad \Longrightarrow \quad \bm{x}(t) \in \Delta^{N-1}, \quad \forall t \geq 0.
    \end{equation}
    Equivalently, if $\bm{x}(0) \in \Delta^{N-1}$, then
    \begin{equation}\label{eq:Simplex_Invariant_COMPONENT} 
        \bm{x}(t) \geq 0 \quad \text{and} \quad \bm{e}^{\mathsf{T}} \bm{x}(t) = 1, \quad \forall t \geq 0,
    \end{equation}
    where here, and throughout this work, vector inequalities are to be intended component-wise. Moreover, every face of the simplex, i.e. each subsimplex of the form
    \begin{equation}\label{eq:Faces_simplices}
        \Delta_{\mathcal{I}}^{N-1} = \left\{\bm{\xi} = [\xi_1,\dots,\xi_N]^\mathsf{T} \in \Delta^{N-1} \ : \ \xi_i = 0 \ \ \forall i \in \mathcal{I} \subset \{1,\dots,N\} \right\},
    \end{equation}
    is positively invariant as well.
    
    \item The invariance of the probability simplex naturally leads to a classification of the equilibria of~\eqref{eq:replicator} according to their location within~$\Delta^{N-1}$, as detailed below.
    Let $\bm{e}^{(k)} = [e_1^{(k)}, \dots, e_N^{(k)}]^\mathsf{T} \in \Delta^{N-1}$, for $k = 1,\dots,N$, denote the $k$-th canonical basis vector of $\mathbb{R}^N$, that is,
    \begin{equation}\label{eq:e_i-Def}
        e_k^{(k)} = 1 \qquad \text{and} \qquad e_l^{(k)} = 0, \quad \text{for all } l \in \{1,\dots,N\} \setminus \{k\}.
    \end{equation}
    Each of these vectors $\bm{e}^{(k)}$ is an equilibrium of the replicator system~\eqref{eq:replicator}, in the sense that
    \begin{equation}\label{eq:equilibria_properties}
        \bm{x}(0) = \bm{e}^{(k)} \qquad \Longrightarrow \qquad \bm{x}(t) = \bm{e}^{(k)}, \quad \forall t \geq 0.
    \end{equation}
    Such equilibria are referred to as corner equilibria, since they lie at the vertices of the simplex.
    In addition, there may exist other equilibria within the interior or on the faces of $\Delta^{N-1}$. These are the so-called internal equilibria, given by points $\bm{x}^* \in \Delta^{N-1}$ satisfying
    \begin{equation}\label{eq:inter_equilibria_properties}
        \bm{e}^\mathsf{T} \bm{x}^* = 1 \qquad \text{and} \qquad f_j(\bm{x}^*) = f^*, \quad \text{for all } j = 1,\dots,N,
    \end{equation}
    for some constant $f^* \in \mathbb{R}$.
    
    \item A distinctive analytical property of the replicator system is the \emph{quotient rule}, which describes the evolution of the ratio between any two components. For all $i,j \in \{1,\dots,N\}$ with $x_j(t) > 0$, it holds that
    \begin{equation}\label{eq:quotient_rule}
        \frac{d}{dt} \left( \frac{x_i(t)}{x_j(t)} \right) = \frac{x_i(t)}{x_j(t)} \left( f_i(\bm{x}(t)) - f_j(\bm{x}(t)) \right).
    \end{equation}
    This shows that the relative growth rate between components $i$ and $j$ depends solely on their fitness difference.
\end{itemize}

The structural features outlined above often reflect fundamental physical or biological principles inherent in replicator systems. For this reason, particular emphasis will be placed, in Section \ref{sec:Method}, on the development of a numerical scheme that unconditionally preserve such properties in the discrete-time setting.

\section{The normalized rational integrator}\label{sec:Method}
Let $\{t_n, \ n=0,\dots,M\},$ with $M \in \mathbb{N}$, denote a (possibly nonuniform) discretization of the integration interval $[0,T]$, satisfying
\begin{equation}\label{eq:Stepsizes_variation}
    t_0 = 0,\quad t_M = T\quad \text{and} \quad t_{n+1} = t_n + h_n, \quad \text{with } h_n \in \mathbb{R}^+, \quad n = 0, \dots, M-1.
\end{equation}
To approximate the solution of the replicator system \eqref{eq:replicator}, we introduce the following numerical scheme
\begin{equation}\label{eq:Rational_Integrator}
    \begin{split}
        &\varphi_i^n=\dfrac{h_n}{2} \left( f_i(\bm{x}^n) - \sum_{j=1}^{N} x_j^n f_j(\bm{x}^n) \right), \qquad  \qquad \qquad 
        \rho_i^n =\dfrac{(\varphi_i^n)^2+6\varphi_i^n+12}{(\varphi_i^n)^2-6\varphi_i^n+12},\\
        &x_i^{n+\frac{1}{2}}= (x_i^n \, \rho_i^n)\left(\displaystyle\sum_{j=1}^N x_j^n \, \rho_j^n\right)^{-1}, \\
        &\varphi_i^{n+\frac{1}{2}}=h_n \left( f_i(\bm{x}^{n+\frac{1}{2}}) - \sum_{j=1}^{N} x_j^{n+\frac{1}{2}} f_j(\bm{x}^{n+\frac{1}{2}}) \right), \quad
        \rho_i^{n+\frac{1}{2}}=\dfrac{(\varphi_i^{n+\frac{1}{2}})^2+6\varphi_i^{n+\frac{1}{2}}+12}{(\varphi_i^{n+\frac{1}{2}})^2-6\varphi_i^{n+\frac{1}{2}}+12},\\
        &x_i^{n+1}= (x_i^{n} \, \rho_i^{n+\frac{1}{2}})\left(\displaystyle\sum_{j=1}^N x_j^{n} \, \rho_j^{n+\frac{1}{2}}\right)^{-1}, \\
    \end{split}
\end{equation}
where $\bm{x}^0=\bm{x}(0)$ is given and $\bm{x}^n = [x_1^n, \dots, x_N^n]^\mathsf{T} \approx \bm{x}(t_n)$ for each $n = 1, \dots, M$.

Throughout this paper we refer, when needed, to the following equivalent compact reformulation of \eqref{eq:Rational_Integrator}
\begin{equation}\label{eq:Compact_Rational_Integrator}
    \begin{split}
        &\bm{x}^{n+\frac{1}{2}} = \dfrac{ \bm{x}^n \odot \bm{R}\left(\dfrac{h_n}{2} \left( \bm{f}(\bm{x}^n) - \left(\bm{f}(\bm{x}^n)^\mathsf{T}\bm{x}^n \right) \bm{e} \right)\right) }{ \bm{e}^{\mathsf{T}} \left( \bm{x}^n \odot \bm{R}\left(\dfrac{h_n}{2} \left( \bm{f}(\bm{x}^n) - \left(\bm{f}(\bm{x}^n)^\mathsf{T}\bm{x}^n \right) \bm{e} \right)\right) \right) }, \\[1.2ex]
        &\bm{x}^{n+1} = \dfrac{ \bm{x}^{n} \odot \bm{R} \left( h_n \left( \bm{f}(\bm{x}^{n+\frac{1}{2}}) - \left( \bm{f}(\bm{x}^{n+\frac{1}{2}})^{\mathsf{T}} \bm{x}^{n+\frac{1}{2}} \right) \bm{e} \right) \right) }{ \bm{e}^{\mathsf{T}}  \left( \bm{x}^{n} \odot \bm{R} \left( h_n \left( \bm{f}(\bm{x}^{n+\frac{1}{2}}) - \left( \bm{f}(\bm{x}^{n+\frac{1}{2}})^{\mathsf{T}} \bm{x}^{n+\frac{1}{2}} \right) \bm{e} \right) \right) \right) },
    \end{split}
\end{equation}
where $\odot$ denotes the component-wise Hadamard product and
\begin{equation*}
    \bm{f}(\bm{x})=[f_1(\bm{x}),\dots,f_N(\bm{x})]^\mathsf{T}\in \mathbb{R}^N \quad \text{and} \quad \bm{R}(\bm{x})=[\rho(x_1),\dots,\rho(x_N)]^\mathsf{T}\in \mathbb{R}^N,
\end{equation*}
with $\rho(z)=(z^2+6z+12)(z^2-6z+12)^{-1}$ for $z\in \mathbb{R}.$ The two-stage formulation~\eqref{eq:Compact_Rational_Integrator} is fully explicit, yielding an efficient and easy-to-implement algorithm. The next sections are devoted to its numerical and dynamical analysis.

\subsection{Structure-preserving properties}\label{subsec:Structural_Preserviation}
The normalized rational scheme \eqref{eq:Rational_Integrator} is specifically designed to preserve, with no limitations on the steplength, the fundamental structural features of the replicator dynamics. The forthcoming results establish the dynamical consistency, as defined in \cite{Dynamic,Dang}, of the proposed discretization. 

\begin{theorem}[\textbf{Simplex-confined discrete dynamics}]\label{thm:Preservazione_Simplesso}
    Let the initial value $\bm{x}^0$ belong to the probability simplex $\Delta^{N-1}$ defined in~\eqref{eq:simplex_def} and let $\{\bm{x}^n\}_{n\in \mathbb{N}_0},$ $\{\bm{x}^{n+\frac{1}{2}}\}_{n\in \mathbb{N}_0}$ be the sequences generated by the numerical method~\eqref{eq:Compact_Rational_Integrator}. Then, for any choices of positive discretization steps $\{h_n\}_{n\in \mathbb{N}_0},$
    \begin{equation*}
        x_i^n,x_i^{n+\frac{1}{2}} \geq 0 \qquad \text{and} \qquad \textstyle\sum_{i=1}^{N} x_i^n =\textstyle\sum_{i=1}^{N} x_i^{n+\frac{1}{2}} = 1, \qquad \forall n\in \mathbb{N}_0.
    \end{equation*}
    Furthermore, if $\bm{x}^0$ lies on a face of the simplex $\Delta^{N-1}_\mathcal{I},$ as defined in \eqref{eq:Faces_simplices}, then $\bm{x}^n,\bm{x}^{n+\frac{1}{2}}\in \Delta^{N-1}_\mathcal{I}$ for all $n\in \mathbb{N}_0.$
\end{theorem}
\begin{proof}
    We preliminarily remark that the polynomials \( p_{\pm}(z) = z^2 \pm 6z + 12 \) satisfy
    \begin{equation*}
        p_{\pm}(z) = (z \pm 3)^2 + 3,
    \end{equation*}
    which implies that $p_{\pm}(z) > 0$ for all $z \in \mathbb{R}$. Hence, the rational function $\rho(z)=p_+(z) / p_-(z)$ and the vector-valued mapping $R(\bm{x})=[\rho(x_1),\dots,\rho(x_N)]^\mathsf{T}$ are well defined and strictly positive on the real line and on $\mathbb{R}^N,$ respectively. \\
    We proceed by mathematical induction to prove that the statements $\bm{x}^n,\bm{x}^{n+\frac{1}{2}}\geq 0$ and $\textstyle\sum_{i=1}^{N} x_i^n =\textstyle\sum_{i=1}^{N} x_i^{n+\frac{1}{2}} = 1$ hold true for all $n\in\mathbb{N}_0.$ The base case $n=0$ naturally follows from the assumptions and the positivity of $\bm{R}$. Consider $n\geq0$ and assume that the properties are verified for each $0\leq m \leq n.$ Then, from \eqref{eq:Compact_Rational_Integrator},
    \begin{equation*}
        \bm{R}(\bm{x}^n)\geq 0 \, \implies \, \bm{x}^{n+\frac{1}{2}} \geq 0 \quad \wedge \quad \bm{R}(\bm{x}^{n+\frac{1}{2}})\geq 0 \, \implies \, \bm{x}^{n+1}\geq0,
    \end{equation*}
    that yields the non-negativity result. Furthermore, straightforward manipulations of \eqref{eq:Compact_Rational_Integrator} lead to $\bm{e}^{\mathsf{T}}  \bm{x}^n=\bm{e}^{\mathsf{T}}  \bm{x}^{n+\frac{1}{2}}=1.$ Analogously, if we assume that $\bm{x}^m\in \Delta^{N-1}_\mathcal{I}$ for each $0\leq m \leq n,$ then from the second and the last equations in \eqref{eq:Rational_Integrator} it follows that $x_i^{n+\frac{1}{2}}=x_i^{n+1}=0$ for all $i\in\mathcal{I},$ which completes the proof.
\end{proof}

Theorem \ref{thm:Preservazione_Simplesso} shows that the rational integrator \eqref{eq:Rational_Integrator} preserves the structure of the probability simplex, ensuring the validity of \eqref{eq:Simplex_Invariant} at the discrete-time level, regardless of the chosen stepsizes. Equivalently, in compliance with \eqref{eq:Simplex_Invariant_COMPONENT}, the scheme~\eqref{eq:Rational_Integrator} is unconditionally non-negative and for the numerical solution the linear invariant
\begin{equation}\label{eq:discrete_Linear_Invariant}
    \bm{e}^{\mathsf{T}}  \bm{x}^{n+\frac{1}{2}}=\bm{e}^{\mathsf{T}}  \bm{x}^n=1, \qquad \forall n \in \mathbb{N}_0,
\end{equation}
holds true. In addition, all faces of the simplex are invariant under the discrete dynamics, thereby preserving the geometric structure of the continuous system.

Beyond geometric invariance, another crucial aspect of dynamical consistency lies in the preservation of equilibrium configurations. The following results show that the rational integrator unconditionally reproduces both corner and internal equilibria of the continuous-time system \eqref{eq:replicator}.

\begin{theorem}[\textbf{Corner equilibria preservation}]\label{thm:Corner_equilibria}
     Assume that the initial value satisfies $\bm{x}^0=\bm{e}^{(k)},$ for any $k\in \{1,\dots,N\},$ with $\bm{e}^{(k)}$ defined in \eqref{eq:e_i-Def}. Then, for any choices of positive steplenghts $\{h_n\}_{n\in \mathbb{N}_0},$ the properties
     \begin{equation*}
         x_k^{n}=1 \qquad \text{and} \qquad x_l^{n}=0, \quad \text{for each} \quad l\in\{1,\dots,N\}\setminus\{k\}, \qquad  \forall n\in \mathbb{N}_0,
    \end{equation*}
    hold true. Equivalently, $\bm{x}^n=\bm{x}^0$ for all $n\in\mathbb{N}_0$ and the simplex corner $\bm{e}^{(k)}$ is a fixed point for the discrete equation \eqref{eq:Rational_Integrator}. 
\end{theorem}
\begin{proof}
    Standard induction arguments yields the result. As a matter of fact, assuming that the statement holds true for an integer $n\geq 0$ leads to 
    \begin{equation*}
        \begin{split}
        & x_l^{n}=0, \quad l\in\{1,\dots,N\}\setminus\{k\}\implies x_l^{n+\frac{1}{2}}=x_l^{n+1}=0, \quad l\in\{1,\dots,N\}\setminus\{k\}, \\
        & x_k^{n}=1 \implies \varphi^n_k=0 \implies \rho_k^{n}=1 \implies\sum_{j=1}^Nx_j^n\rho_j^n=\rho_k^{n}=1 \implies x_k^{n+\frac{1}{2}}=x_k^n=1
        \end{split}
    \end{equation*}
    and therefore to
    \begin{equation*}
        \varphi^{n+\frac{1}{2}}_k=0 \implies \rho_k^{n+\frac{1}{2}}=1 \implies\sum_{j=1}^Nx_j^{n+\frac{1}{2}}\rho_j^{n+\frac{1}{2}}=\rho_k^{n+\frac{1}{2}}=1 \implies x_k^{n+1}=x_k^n=1,
    \end{equation*}
    which completes the proof.
\end{proof}

\begin{theorem}[\textbf{Internal equilibria preservation}]\label{thm:Internal_Equilibria}
    Assume that the initial value $\bm{x}^0=\bm{x}^*$ exhibits the properties in \eqref{eq:inter_equilibria_properties}. Then, independently of the positive discretization steps $\{h_n\}_{n\in \mathbb{N}_0},$ 
    \begin{equation*}
        x_i^{n}=x_i^*, \qquad i=1,\dots,N, \qquad \forall n \in \mathbb{N}_0
    \end{equation*}
    and therefore $\bm{x}^*$ is a fixed point for the numerical method \eqref{eq:Rational_Integrator}.
\end{theorem}
\begin{proof}
    We proceed by mathematical induction. The base case, i.e. $n=0,$ follows from the the assumptions. Consider now an integer $n\geq 0$ and assume that $\bm{x}^n=\bm{x}^*.$ It then follows that $\bm{f}(\bm{x}^n)=f^*\bm{e}$ and, from \eqref{eq:inter_equilibria_properties} and \eqref{eq:Compact_Rational_Integrator},  
    \begin{equation*}
        \begin{split}
            \bm{f}(\bm{x}^n)^\mathsf{T}\bm{x}^n=f^* &\implies \bm{R}\left(\dfrac{h_n}{2} \left( \bm{f}(\bm{x}^n) - \left(\bm{f}(\bm{x}^n)^\mathsf{T}\bm{x}^n \right) \bm{e} \right)\right)=\bm{R}(0)=\bm{e} \\
            &\implies \bm{x}^{n+\frac{1}{2}}=\bm{x}^n=\bm{x}^*.
        \end{split}
    \end{equation*}
    Taking advantage of the last equality, similar arguments yield the relation $\bm{R} \left( h_n \left( \bm{f}(\bm{x}^{n+\frac{1}{2}}) - \left( \bm{f}(\bm{x}^{n+\frac{1}{2}})^{\mathsf{T}} \bm{x}^{n+\frac{1}{2}} \right) \bm{e} \right) \right)=\bm{e},$ which then results in $\bm{x}^{n+1}=\bm{x}^n=\bm{x}^*.$
\end{proof}

Theorems~\ref{thm:Corner_equilibria} and~\ref{thm:Internal_Equilibria} ensure that the equilibria of the continuous-time system~\eqref{eq:replicator} are preserved by the numerical scheme~\eqref{eq:Rational_Integrator}, regardless of the discretization stepsizes. Combined with Theorem~\ref{thm:Preservazione_Simplesso}, these outcomes confirm the structure-preserving nature of the proposed discretization and thus render it well suited for the approximation of replicator dynamics on the probability simplex.

\section{Error Analysis and Convergence}\label{Sec:Error}
Here, we perform an analysis of the approximation error yielded by the two-stage discretization \eqref{eq:Rational_Integrator} and prove the second-order convergence of the corresponding numerical solution. As a first step, we address the numerical consistency through the local truncation error, here denoted by
\begin{adjustwidth}{-0.4in}{0in}
\begin{equation}\label{eq:Local_Error}
        \begin{split}
        &\delta_i(h_n;t_{n+1}) \!= \!x_i(t_{n+1})-\dfrac{x_i(t_n) \ \rho \! \left( \!  h_n \! \left( \! f_i(\bm{\eta}(\bm{x}(t_n))) - \displaystyle \sum_{j=1}^{N} \eta_j(\bm{x}(t_n)) f_j(\bm{\eta}(\bm{x}(t_n))) \right) \! \right)}{\displaystyle \sum_{j=1}^{N} x_j(t_n)  \rho \! \left( \! h_n \!\! \left( \! f_j(\bm{\eta}(\bm{x}(t_n))) \! - \!\! \displaystyle \sum_{k=1}^{N} \eta_k(\bm{x}(t_n)) f_k(\bm{\eta}(\bm{x}(t_n))) \! \right) \!\! \right) },
    \end{split}
\end{equation}
\end{adjustwidth}
$i=1,\dots,N,$where the components of $\eta(\bm{x}(t_n))=[\eta_1(\bm{x}(t_n)),\dots,\eta_N(\bm{x}(t_n))]^{\mathsf{T}}$ are defined as
\begin{equation}\label{eq:eta}
        \eta_i(\bm{x}(t_n))=\dfrac{x_i(t_{n}) \ \rho \! \left( \! \dfrac{h_n}{2} \! \left( \! f_i(\bm{x}(t_{n})) - \displaystyle \sum_{j=1}^{N} x_j(t_{n}) f_j(\bm{x}(t_{n})) \right) \! \right)}{\displaystyle \sum_{j=1}^{N} x_j(t_{n}) \ \rho \! \left( \! \dfrac{h_n}{2} \! \left( \! f_j(\bm{x}(t_{n})) - \displaystyle\sum_{k=1}^{N} x_k(t_{n}) f_k(\bm{x}(t_{n})) \right) \! \right) }. 
\end{equation}

\begin{theorem}\label{thm:Consistency}
    Assume that the given functions describing problem \eqref{eq:replicator} are twice continuously differentiable on the probability simplex $\Delta^{N-1}$ defined in \eqref{eq:simplex_def}. Then, the normalized rational integrator \eqref{eq:Rational_Integrator} is consistent with \eqref{eq:Compact_Rational_Integrator}, of order $2.$
\end{theorem}
\begin{proof}
    For each \( i = 1, \dots, N \), define the scalar functions
    \begin{equation}\label{eq:F_defn}
        F_i(\bm{\xi}) = f_i(\bm{\xi}) - \textstyle\sum_{j=1}^N x_j f_j(\bm{\xi}), \qquad \qquad \bm{\xi}\in\Delta^{N-1}.
    \end{equation}
    Since $\rho(z) = (z^2 + 6z + 12)(z^2 - 6z + 12)^{-1}$ for $z \in \mathbb{R}$, a straightforward Taylor expansion yields
    \begin{equation*}
        \rho\left( \frac{h_n}{2} F_i(\bm{x}(t_n)) \right) = 1 + \frac{h_n}{2} F_i(\bm{x}(t_n)) + \frac{h_n^2}{8} F_i^2(\bm{x}(t_n)) + \mathcal{O}(h_n^3),
    \end{equation*}
    for each $i=1,\dots,N.$ Moreover, from the identities \( \sum_{j=1}^N x_j(t_n) = 1 \) and \( \sum_{j=1}^N x_j(t_n) F_j(\bm{x}(t_n)) = 0 \), we deduce
        \begin{equation*}
        \left( \sum_{j=1}^N x_j(t_n) \, \rho\left( \frac{h_n}{2} F_j(\bm{x}(t_n)) \right) \right)^{-1} = 1 - \frac{h_n^2}{8} \sum_{j=1}^N F_j^2(\bm{x}(t_n)) \, x_j(t_n) + \mathcal{O}(h_n^3).
        \end{equation*}
    Hence, recalling the definition in \eqref{eq:eta}, it follows that
    \begin{equation}\label{eq:temp_1}
        \begin{split}
            \eta_i(\bm{x}(t_n))=&x_i(t_{n}) \left(1+\frac{h_n}{2}F_i(\bm{x}(t_n))+\frac{h_n^2}{8}\left(F_i^2(\bm{x}(t_{n}))-\sum_{j=1}^NF_j^2(\bm{x}(t_{n}))x_j(t_n) \right)\right)\\
            &+\mathcal{O}(h_n^3), \qquad \qquad \qquad \qquad \qquad \qquad\qquad \qquad \qquad \qquad i=1,\dots,N.
        \end{split}
    \end{equation}
    In particular, one obtains
        \begin{equation*}
        \bm{\eta}(\bm{x}(t_n))-\bm{x}(t_n) = \frac{h_n}{2} \, \dfrac{d \bm{x}}{d t}(t_n) + \mathcal{O}(h_n^2),
        \end{equation*}
    which in turn implies
    \begin{equation*}
        \begin{split}
            \rho \left( h_n F_i(\bm{\eta}(\bm{x}(t_n))) \right) &= 1 + h_n F_i(\bm{x}(t_n)) + \frac{h_n^2}{2} \left( \nabla F_i^\mathsf{T}(\bm{x}(t_n)) \cdot \dfrac{d \bm{x}}{d t}(t_n) + F_i^2(\bm{x}(t_n)) \right) \\
            & \quad + \mathcal{O}(h_n^3),
        \end{split}
    \end{equation*}
    and, recalling $\frac{dx_i}{dt}(t_n)=x_i(t_n)F_i(\bm{x}(t_n)),$
    \begin{equation}\label{eq_tmp_1}
        x_i(t_n) \, \rho\left( h_n F_i(\bm{x}(t_n)) \right) = x_i(t_{n+1}) + \mathcal{O}(h_n^3),
    \end{equation}
    for all \( i = 1, \dots, N \). Therefore, we conclude
    \begin{equation}\label{eq:Normal_ordine_3}
        \left( \sum_{j=1}^N x_j(t_n) \, \rho\left( h_n F_j(\bm{\eta}(\bm{x}(t_n))) \right) \right)^{-1} = 1 + \mathcal{O}(h_n^3),
    \end{equation}
    so that \( \delta_i(h_n; t_{n+1}) = \mathcal{O}(h_n^3) \) for all \( i = 1, \dots, N \), which yields the result.
\end{proof}

The straightforward extension of well-established results assures the quadratic convergence of the scheme \eqref{eq:Rational_Integrator}.

\begin{theorem}\label{thm:Convergence}
    Let $\bm{x}(t)$ be the continuous-time solution to \eqref{eq:replicator} on the interval $[0,T]$ and let $\{\bm{x}^n\}_{0\leq n \leq M}$ be its approximation computed by the normalized rational integrator \eqref{eq:Compact_Rational_Integrator}, with stepsizes $\{h_n\}_{0\leq n < M}$ satisfying \eqref{eq:Stepsizes_variation}. Suppose that the known functions defining \eqref{eq:Compact_Rational_Integrator} are twice continuously differentiable on the probability simplex $\Delta^{N-1}.$ Then, there exists a constant $C>0$ such that
    \begin{equation*}
        \max_{n=0,\dots,M}\|\bm{x}(t_n)-\bm{x}^n\|\leq C h^2,  \qquad \qquad h=\max_{i=0,\dots,M-1} h_i,
    \end{equation*}
    and hence the method \eqref{eq:Compact_Rational_Integrator} is second-order convergent.
\end{theorem}
\begin{proof}
    Since $\bm{R}\in C^\infty(\Delta^{N-1}),$ from the assumptions, the increment functions at the right-hand sides of \eqref{eq:Compact_Rational_Integrator} are twice continuously differentiable. Moreover, by Theorem~\ref{thm:Consistency}, the method is consistent of order 2. Therefore, the result follows from \cite[Theorem 3.6.]{Hairer}.
\end{proof}
The numerical method \eqref{eq:Rational_Integrator} is constructed as a composition of a two-stage rational approximation with a normalization step. The latter is introduced to ensure that the discrete-time dynamics remains confined to the probability simplex for any stepsize (cf. Theorem~\ref{thm:Preservazione_Simplesso}). On the other hand, equation~\eqref{eq:Normal_ordine_3} reveals that the normalization procedure corresponds to a third-order perturbation of the identity operator, which does not alter the second-order accuracy of the underlying rational approximation established by \eqref{eq_tmp_1}. As a result, the structure-preserving scheme \eqref{eq:Rational_Integrator} exhibits quadratic convergence.

\begin{remark}
    Positivity-preserving methods of second order are rare, especially in the context of nonlinear systems. The NSFD schemes recently introduced in~\cite{Gupta,Kojouharov} attain second-order accuracy and preserve local asymptotic stability of equilibria for scalar autonomous ODEs, yet they do not guarantee positivity. A quadratically convergent positivity-preserving method is proposed in~\cite{Hoang_Second}, although still restricted to scalar problems. In contrast, our approach achieves second-order accuracy while ensuring positivity and dynamical consistency for the multidimensional replicator dynamics.
\end{remark}

\subsection{The discrete quotient rule}
The continuous-time \emph{quotient rule} \eqref{eq:quotient_rule}, which governs the evolution of the relative growth $x_i(t)/x_j(t)$ through the fitness difference $\hat{f}_{ij}(\bm{x}) := f_i(\bm{x}) - f_j(\bm{x})$, can be reformulated in integral form as follows
\begin{equation}\label{eq:quotient_rule_INTEGRAL}
    \frac{x_i(t_{n+1})}{x_j(t_{n+1})} = \frac{x_i(t_n)}{x_j(t_n)} \exp \left\{ \int_{t_n}^{t_{n+1}} \hat{f}_{ij}(\bm{x}(\tau)) \, d\tau \right\}.
\end{equation}
This identity suggests that the discrete-time ratio $x_i^{n+1}/x_j^{n+1}$ should replicate the exponential structure on the right-hand side of~\eqref{eq:quotient_rule_INTEGRAL}. The result below shows that the rational integrator~\eqref{eq:Rational_Integrator} achieves this up to second-order accuracy in the stepsize.
\begin{theorem}[\textbf{Discrete quotient rule}]\label{thm:Discrete_Quotient_Rule}
Assume that the functions defining the replicator system~\eqref{eq:replicator} are twice continuously differentiable on the probability simplex~$\Delta^{N-1}$, defined in~\eqref{eq:simplex_def}. Let $\{\bm{x}^n\}$ be the sequence generated by the numerical scheme~\eqref{eq:Rational_Integrator} with stepsizes $\{ h_n \}$. Then, for any $i,j \in \{1,\dots,N\}$ with $x_j^0 > 0$, the following relation holds 
\begin{equation}\label{eq:quotient_rule_Discrete}
    \frac{x_i^{n+1}}{x_j^{n+1}} = \frac{x_i^n}{x_j^n} \left( \, \exp\left\{ h_n \, \hat{f}_{ij}(\bm{x}^n) \right\} + \mathcal{O}(h_n^2) \, \right), \qquad \text{ for all } n\geq0,
\end{equation}
where $\hat{f}_{ij}(\bm{x}) := f_i(\bm{x}) - f_j(\bm{x})$ denotes the fitness difference.
\end{theorem}
\begin{proof}
   Theorem~\ref{thm:Preservazione_Simplesso} guarantees that $\bm{x}^{n+\frac{1}{2}} \in \Delta^{N-1}$. As a consequence,  $\varphi_i^{n+\frac{1}{2}} = \mathcal{O}(h_n)$. Recalling that the rational function $\rho(z)$ coincides with the $[2/2]$ Padé approximant of the exponential function (see~\cite[Sections~1.2 and~1.4]{Pade_Book}), we have
    \begin{equation*}
        \rho(z) = \frac{1 + \frac{1}{2} z + \frac{1}{12} z^2}{1 - \frac{1}{2} z + \frac{1}{12} z^2} = \exp(z) + \mathcal{O}(z^5),
    \end{equation*}
    and thus, for each $i = 1,\dots,N$,
    \begin{equation*}
        \rho_i^{n+\frac{1}{2}} = \rho\left( \varphi_i^{n+\frac{1}{2}} \right) = \exp\left( \varphi_i^{n+\frac{1}{2}} \right) + \mathcal{O}(h_n^5).
    \end{equation*}
    Now, assuming $x_j^0 > 0$ (which ensures $x_j^n > 0$ for all $n$), we obtain from \eqref{eq:Rational_Integrator}
    \begin{equation*}
        \begin{split}
            \frac{x_i^{n+1}}{x_j^{n+1}} &= \frac{x_i^n}{x_j^n} \left( \, \exp\left\{ \varphi_i^{n+\frac{1}{2}}-\varphi_j^{n+\frac{1}{2}} \right\} + \mathcal{O}(h_n^5) \, \right) \\
            &= \frac{x_i^n}{x_j^n} \left( \, \exp\left\{ h_n \hat{f}_{ij}(\bm{x}^n) + h_n (\nabla \hat{f}_{ij}(\bm{x}^n) )^\mathsf{T}(\bm{x}^{n+\frac{1}{2}}-\bm{x}^n)\right\} + \mathcal{O}(h_n^2) \, \right),
        \end{split}
    \end{equation*}
    for all $i=1,\dots,N.$ Finally, from Theorem \ref{thm:Convergence}, $\bm{x}^n=\bm{x}(t_n)+\mathcal{O}(h_n^2)$ and from \eqref{eq:temp_1} 
    \begin{equation*}
        \begin{split}
            x_i^{n+\frac{1}{2}}-x_i^n=x_i^n \left( \rho_i^n\left(\displaystyle\sum_{j=1}^N x_j^n \, \rho_j^n\right)^{-1}-1\right)=\mathcal{O}(h_n), \qquad \qquad i=1,\dots, N,
        \end{split}
    \end{equation*}
    which implies $\bm{x}^{n+\frac{1}{2}} - \bm{x}^n = \mathcal{O}(h_n)$, yielding \eqref{eq:quotient_rule_Discrete}.
\end{proof}
We remark that the relation \eqref{eq:quotient_rule_Discrete} constitutes a discrete analogue of the continuous identity \eqref{eq:quotient_rule_INTEGRAL}. As a matter of fact, under the aforementioned smoothness assumptions, the integral in \eqref{eq:quotient_rule_INTEGRAL} admits the approximation
\begin{equation*}
    \int_{t_n}^{t_{n+1}} \hat{f}_{ij}(\bm{x}(\tau)) \, d\tau = h_n \, \hat{f}_{ij}(\bm{x}(t_n)) + \mathcal{O}(h_n^2)= h_n \, \hat{f}_{ij}(\bm{x}^n) + \mathcal{O}(h_n^2),
\end{equation*}
which corresponds to the left-point rectangular quadrature rule.

\subsection{Adaptive time-stepping via local error estimation}\label{Subsec:VariableH}
To improve the efficiency of the rational integrator~\eqref{eq:Rational_Integrator}, we develop an adaptive time-stepping strategy based on a local error Proportional Integral (PI) controller~\cite{Gustafsson_1991,PI2}. The structure-preserving properties of the integrator guarantee that, regardless of the time-step variation, the discrete dynamics remain consistent with the qualitative features of the continuous system. This allows us to focus on accuracy and cost-efficiency through step refinement, without compromising the underlying dynamical integrity. To this end, we introduce the auxiliary scheme
\begin{equation}\label{eq:auxiliary_scheme}
    \chi_i^{n+1}=(x_i^n \, \tilde{\varrho}_i^n)\left(\displaystyle\sum_{j=1}^N x_j^n \, \tilde{\varrho}_j^n\right)^{-1}, \qquad \tilde{\varrho}_i^n=\dfrac{(2\varphi_i^n+3)^2+3}{(2\varphi_i^n-3)^2+3}, \qquad i=1,\dots,N,
\end{equation}
for $n=0,\dots,M-1$, with $\varphi_i^n$ defined in \eqref{eq:Rational_Integrator}. The structure-preserving properties of \eqref{eq:auxiliary_scheme} follow from the same arguments developed in Section~\ref{subsec:Structural_Preserviation}. Here, we prove that the ancillary scheme \eqref{eq:auxiliary_scheme} provides a first-order approximation of the solution to \eqref{eq:replicator} and that the quantity $|\chi_i^{n+1}-x_i^{n+1}|$ provides a computable surrogate for the local truncation error $\delta_i(h_n;t_{n+1})$ associated with the normalized rational integrator \eqref{eq:Compact_Rational_Integrator}. 
\begin{theorem}\label{thm:local_Estimate_Controller}
    Assume that the given functions describing problem \eqref{eq:replicator} are continuously differentiable on the probability simplex $\Delta^{N-1}$ defined in \eqref{eq:simplex_def}. Then, the auxiliary method is linearly convergent. Moreover, it asymptotically reproduces the local truncation error of the normalized rational integrator \eqref{eq:Compact_Rational_Integrator}.
\end{theorem}
\begin{proof}
       Let $\bm{\xi}\in\Delta^{N-1}$ and define $\bm{\chi}(\bm{\xi})=[\chi_1(\bm{\xi}),\dots,\chi_N(\bm{\xi})]^\mathsf{T}$, where $F_i(\bm{\xi})$ is given by \eqref{eq:F_defn} and 
    \begin{equation*}
        \chi_i(\bm{\xi})=\dfrac{\xi_i \ \rho  \left( h_n F_i(\bm{\xi})\right)}{\sum_{j=1}^{N} \xi_j \rho  \left( h_n F_j(\bm{\xi})\right) }, \qquad \qquad i=1,\dots,N.
    \end{equation*}
        Proceeding as in the proof of Theorem \ref{thm:Consistency}, we obtain $\chi_i(\bm{x}(t_n))=x_i(t_n)\cdot (1 + h_n F_i(\bm{x}(t_n))+\mathcal{O}(h_n^2)) \cdot (1+\mathcal{O}(h_n^2)),$ $i=1,\dots,N$ and therefore
    \begin{equation*}
        \bm{\chi}(\bm{x}(t_n))=\bm{x}(t_n) + \frac{h_n}{2} \, \dfrac{d \bm{x}}{d t}(t_n) + \mathcal{O}(h_n^2)=\bm{x}(t_{n+1})+\mathcal{O}(h_n^2),
    \end{equation*}
        which implies that \eqref{eq:auxiliary_scheme} is consistent and convergent of order one. Furthermore, it directly follows from the definition \eqref{eq:Local_Error} that
    \begin{equation*}
        \begin{split}
        &\chi_i(\bm{x}(t_n))-\dfrac{x_i(t_n) \ \rho \! \left( \!  h_n \! \left( \! f_i(\bm{\eta}(\bm{x}(t_n))) - \displaystyle \sum_{j=1}^{N} \eta_j(\bm{x}(t_n)) f_j(\bm{\eta}(\bm{x}(t_n))) \right) \! \right)}{\displaystyle \sum_{j=1}^{N} x_j(t_n)  \rho \! \left( \! h_n \!\! \left( \! f_j(\bm{\eta}(\bm{x}(t_n))) \! - \!\! \displaystyle \sum_{k=1}^{N} \eta_k(\bm{x}(t_n)) f_k(\bm{\eta}(\bm{x}(t_n))) \! \right) \!\! \right) } \\ 
        &= \delta_i(h_n;t_{n+1})+\mathcal{O}(h_n^2), \qquad  \qquad\qquad\qquad\qquad \qquad\qquad \qquad i=1,\dots,N,
        \end{split}
    \end{equation*}
    which completes the proof.
\end{proof}

Having established the approximation and structural properties of the auxiliary scheme~\eqref{eq:auxiliary_scheme}, we now incorporate it into an adaptive time-stepping strategy aimed at controlling the integration accuracy. In the spirit of embedded methods~\cite[Section~II.4]{Hairer}, schemes~\eqref{eq:Rational_Integrator} and~\eqref{eq:auxiliary_scheme} are interpreted as a primary and embedded pair, with the latter furnishing a practical estimate of the local error. To this end, we introduce two vectors of component-wise tolerances, denoted by $\bm{\tau}^{\text{abs}}$ and $\bm{\tau}^{\text{rel}}$, prescribing absolute and relative error bounds. At each time-step $n$, given the tentative approximations $\bm{x}^n$ and $\bm{\xi}^n$ computed by the rational and auxiliary schemes with stepsize $h_n$, the admissibility threshold for the $i$-th component is defined as
\begin{equation*}
    \tau_i = \tau_i^{\text{abs}} + \tau_i^{\text{rel}} \cdot \max\{ x_i^n, \, \xi_i^n \}, \qquad i = 1, \dots, N.
\end{equation*}
The normalized error estimate associated with the current step is then given by
\begin{equation*}
    \varepsilon_n = \sqrt{\frac{1}{N} \sum_{i=1}^N \left( \frac{x_i^n - \xi_i^n}{\tau_i} \right)^2 }.
\end{equation*}
If the condition $\varepsilon_n \leq 1$ is satisfied, the step is accepted; otherwise, it is rejected and recomputed with a smaller stepsize. Independently of acceptance, the stepsize is subsequently updated by means of a PI controller, following the principles outlined in~\cite[Section~IV.2]{Hairer_II}. Specifically, we set
\begin{equation}\label{eq:PI_Controller}
    h_{n+1} = h_n \cdot \min\left\{ \gamma_{\texttt{max}}, \ \max\left\{ \gamma_{\texttt{min}}, \ \gamma \left( \frac{\min_i \tau_i}{\varepsilon_n} \right)^{\alpha - \beta} \! \left( \frac{\varepsilon_{n-1}}{\varepsilon_n} \right)^{\beta} \right\} \right\}, \qquad n\geq 0
\end{equation}
where the controller parameters are chosen as follows
\begin{equation*}
    \alpha = 0.7/r, \quad \beta = 0.4/r, \quad \gamma = (0.25)^{1/r}, \quad r = 2, \quad \gamma_{\texttt{min}} = 0.1, \quad \gamma_{\texttt{max}} = 5.
\end{equation*}
If no prior information is available regarding a suitable initial stepsize $h_0$, the procedure proposed in~\cite[p.140]{Hairer} may be employed. Furthermore, in order to mitigate possible instabilities following step rejections, we additionally impose $\gamma_{\texttt{max}} = 1$ in the iteration immediately succeeding a rejected step.  \\ The adaptive mechanism described above then provides a robust and efficient control of the time-step selection, ensuring that the prescribed accuracy is maintained throughout the integration process. Moreover, since \eqref{eq:auxiliary_scheme} employs the same internal function evaluations as the main integrator \eqref{eq:Rational_Integrator} and requires only elementary operations along with an additional normalization step, it introduces no significant computational efforts.

\section{Numerical experiments}\label{sec:Numerical_Experiments}
This section presents numerical experiments that provide empirical evidence for the theoretical findings established in the previous sections. Representative test cases of the form \eqref{eq:replicator} are here considered and discretized using the normalized rational scheme \eqref{eq:Rational_Integrator}.

\paragraph{Test 1. Convergence order -}\label{test1}\!\!\!\!As a first example we integrate the replicator problem \eqref{eq:replicator} in the case of a simple zero-sum rock-paper-scissor game, for which the fitness function assumes the linear formulation 
\begin{equation}\label{eq:Test 1}
  \bm{f}(\bm{x}(t))=A\bm{x}(t) \qquad \text{with} \qquad A=\begin{bmatrix}
        0 & -1 & 1 \\ 1 & 0 & -1 \\ -1 & 1 & 0
    \end{bmatrix}.
\end{equation}
To assess the experimental convergence behaviour predicted by Theorems~\ref{thm:Convergence} and~\ref{thm:local_Estimate_Controller}, we consider both the normalized rational integrator~\eqref{eq:Rational_Integrator} and the auxiliary scheme~\eqref{eq:auxiliary_scheme}, implemented with a fixed time step $h$. For various choices of $h$, we compute the mean absolute error in the discrete $\ell^2$-norm, denoted by $E_2(h)$ and the corresponding experimental convergence rate $\hat{p}$ as follows
\begin{equation*}
    E_2(h)= \frac{1}{M} \sum_{n=0}^{M} \left\| \bm{y}^{n(ref)} - \bm{y}^n \right\|_2, \qquad \qquad \qquad \hat{p}= \log_2\left( \frac{E_2(h)}{E_2\left(h/2\right)}\right).
\end{equation*}
The reference solution $\bm{y}^{n(ref)}$ is computed via MATLAB’s built-in solver \texttt{ode15s}, with absolute and relative tolerances set to \texttt{AbsTol} $= \varepsilon_{\mathrm{mach}} \approx 2.22 \cdot 10^{-16}$ and \texttt{RelTol} $= 10^{2} \varepsilon_{\mathrm{mach}}$, respectively. As illustrated in Figure~\ref{fig:Order_Test_1} and detailed in Table~\ref{tab:Order_Test_1}, the numerical results confirm the quadratic and linear convergence of schemes~\eqref{eq:Rational_Integrator} and~\eqref{eq:auxiliary_scheme}, respectively.

\begin{figure}[!ht]
    \begin{minipage}{73mm}
    \includegraphics[width=\textwidth]{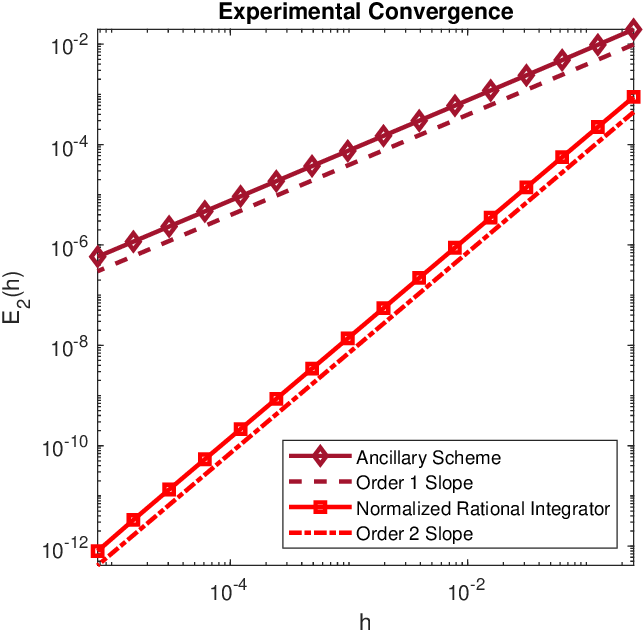}
    \captionsetup{labelformat=empty}
    \caption{}\vspace{-1.5\baselineskip}
    \label{fig:Order_Test_1}
    \end{minipage}
    \marginnote{\textbf{Figure \ref{fig:Order_Test_1}.} \\ Experimental convergence behaviour of the normalized rational integrator~\eqref{eq:Rational_Integrator} and the ancillary scheme~\eqref{eq:auxiliary_scheme} in \hyperref[test1]{Test~1}.}[-5.5\baselineskip]
\end{figure}
\begin{table}[!ht]
    \begin{minipage}{\linewidth}
      \hspace{0.55cm} \begin{tabular}{|c|cc|c|cc|}
            \multicolumn{3}{c|}{\bf Rational Integrator} & \multicolumn{3}{|c}{\bf Ancillary Scheme} \\ \hline
            $h$ & $E_2(h)$ & $\hat{p}$ & $h$ & $E_2(h)$ & $\hat{p}$ \\ \hline
            $2^{-2}$  & $8.88\cdot10^{-4}$  & --    & $2^{-2}$  & $1.97\cdot10^{-2}$ & --    \\ 
            $2^{-3}$  & $2.24\cdot10^{-4}$  & 1.99  & $2^{-3}$  & $9.70\cdot10^{-3}$ & 1.02  \\ 
            $2^{-4}$  & $5.62\cdot10^{-5}$  & 1.99  & $2^{-4}$  & $4.81\cdot10^{-3}$ & 1.01  \\ 
            $2^{-5}$  & $1.41\cdot10^{-5}$  & 2.00  & $2^{-5}$  & $2.40\cdot10^{-3}$ & 1.01  \\ 
            $2^{-6}$  & $3.52\cdot10^{-6}$  & 2.00  & $2^{-6}$  & $1.20\cdot10^{-3}$ & 1.00  \\ 
            $2^{-7}$  & $8.80\cdot10^{-7}$  & 2.00  & $2^{-7}$  & $5.97\cdot10^{-4}$ & 1.00  \\ 
            $2^{-8}$  & $2.20\cdot10^{-7}$  & 2.00  & $2^{-8}$  & $2.98\cdot10^{-4}$ & 1.00  \\ 
            $2^{-9}$  & $5.50\cdot10^{-8}$  & 2.00  & $2^{-9}$  & $1.49\cdot10^{-4}$ & 1.00  \\ 
            $2^{-10}$ & $1.38\cdot10^{-8}$  & 2.00  & $2^{-10}$ & $7.46\cdot10^{-5}$ & 1.00  \\ 
            $2^{-11}$ & $3.44\cdot10^{-9}$  & 2.00  & $2^{-11}$ & $3.73\cdot10^{-5}$ & 1.00  \\ 
            $2^{-12}$ & $8.60\cdot10^{-10}$ & 2.00  & $2^{-12}$ & $1.86\cdot10^{-5}$ & 1.00  \\ 
            $2^{-13}$ & $2.15\cdot10^{-10}$ & 2.00  & $2^{-13}$ & $9.32\cdot10^{-6}$ & 1.00  \\ 
            $2^{-14}$ & $5.36\cdot10^{-11}$ & 2.00  & $2^{-14}$ & $4.66\cdot10^{-6}$ & 1.00  \\ 
            $2^{-15}$ & $1.33\cdot10^{-11}$ & 2.02  & $2^{-15}$ & $2.33\cdot10^{-6}$ & 1.00  \\ 
            $2^{-16}$ & $3.16\cdot10^{-12}$ & 2.07  & $2^{-16}$ & $1.16\cdot10^{-6}$ & 1.00  \\ 
            $2^{-17}$ & $6.70\cdot10^{-13}$ & 2.24  & $2^{-17}$ & $5.82\cdot10^{-7}$ & 1.00  \\ \hline
        \end{tabular}
        \captionsetup{labelformat=empty}
        \caption{}\vspace{-1.5\baselineskip}
        \label{tab:Order_Test_1}
    \end{minipage}
    \marginnote{\textbf{Table \ref{tab:Order_Test_1}.} \\
    Numerical errors and experimental convergence rates for the normalized rational integrator~\eqref{eq:Rational_Integrator} and the ancillary scheme~\eqref{eq:auxiliary_scheme} in \hyperref[test1]{Test~1}.}[-16.5\baselineskip]
\end{table}

For the rock-paper-scissors game governed by the skew-symmetric matrix in \eqref{eq:Test 1}, it is well known that the continuous replicator system~\eqref{eq:replicator} admits periodic trajectories confined to the probability simplex (see, e.g., \cite{RPS_Closed,RPS_Closed_2} and references therein). Thanks to its structure-preserving design, the normalized rational integrator~\eqref{eq:Rational_Integrator} is capable of reproducing this qualitative behaviour, as numerical solutions remain confined within the simplex and closely follow the expected closed orbits even for relatively large time step values. Figure~\ref{fig:Order_Test_Side} shows numerical trajectories computed over the interval $[0,100]$ with fixed stepsize $h = 5 \cdot 10^{-1}$ and multiple initial conditions. Moreover, in agreement with Theorem~\ref{thm:Internal_Equilibria}, the unique interior equilibrium $[1/3,1/3,1/3]^\mathsf{T}$ is exactly preserved and constitutes a fixed point of the scheme~\eqref{eq:Rational_Integrator}.

\begin{figure}[!ht]
    \begin{minipage}{73mm}
        \includegraphics[width=\textwidth]{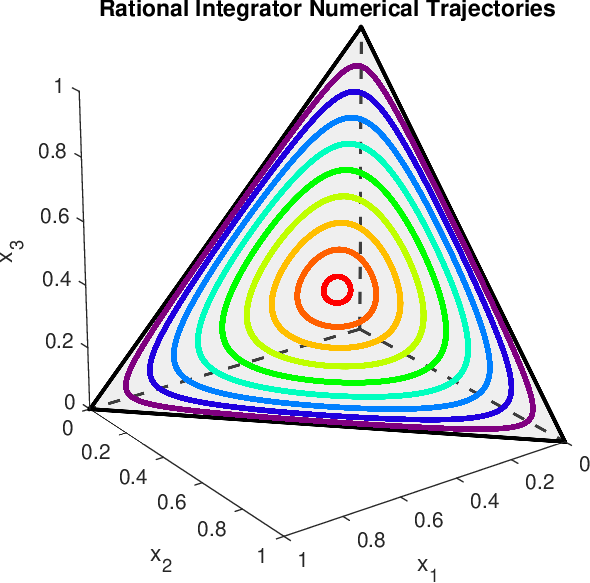}
        \captionsetup{labelformat=empty}
        \caption{}\vspace{-1.5\baselineskip}
        \label{fig:Order_Test_Side}
    \end{minipage}
    \marginnote{\textbf{Figure~\ref{fig:Order_Test_Side}.} \\
    Numerical trajectories for the rock-paper-scissors dynamics in \hyperref[test1]{Test~1}, computed by~\eqref{eq:Rational_Integrator} with $h = 5 \cdot 10^{-1}$ and $T = 10^2$.}[-4.5\baselineskip]
\end{figure}

\paragraph{Test 2. Equilibria preservation and dynamical consistency -}\label{test2}\!\!\!\!We now consider the replicator system \eqref{eq:replicator} with a non-linear fitness function defined by
\begin{equation}\label{eq:Test 2}
    \begin{split}
        &f_1(\bm{x})=\left(10x_3-2\right)^2 \left( \sqrt[10]{\exp(10x_1-1)} + \sin \left(\dfrac{5\pi}{8}x_2\right) + \cosh \left(x_3^2-\dfrac{1}{100}\right) \right), \\
        &f_2(\bm{x})=\log\left(\dfrac{4x_1+1}{3}\right) \left(\dfrac{2+x_1+x_2+x_3}{\log(7)-\log(15)}\right), \\
        &f_3(\bm{x})=3\sin\left(\dfrac{3}{10}-x_2\right)\csc\left(-\dfrac{1}{2}\right) \sqrt{125\, x_1 \, x_2 \, x_3}.
    \end{split}
\end{equation}
It is readily verified that the points $\bm{x}^* = 0.1\cdot[1,\, 8,\, 1]^\mathsf{T}\in\Delta^2$ and $\bm{x}^{**} = 0.1\cdot[5,\, 3,\, 2]^\mathsf{T}\in\Delta^2$ are internal equilibria of the continuous-time dynamics. In what follows, we assess the ability of various time integration methods to preserve the equilibrium $\bm{x}^*$, which is particularly sensitive to numerical perturbations due to the involved structure of the fitness function \eqref{eq:Test 2}. More specifically, we compare the normalized rational integrator~\eqref{eq:Rational_Integrator} with the forward Euler, Heun \cite{heun}, Bogacki–Shampine \cite{bogacki1989}, fourth-order classical Runge–Kutta \cite{Butcher_2016} and Dormand–Prince \cite{dormand1980} methods. All the discretizations are applied with a fixed step size $h=0.2$ over the time interval $[0,35].$ Furthermore, all the simulations are initialized at the equilibrium point so that any observed deviation, measured through the numerical distance
\begin{equation*}
    D^*(t_n) = \sum_{i=1}^3 |x^*_i - x_i^n|, \qquad \qquad n=1,\dots,M,
\end{equation*}
can be unequivocally attributed to the numerical integration method. The results, shown in Figure~\ref{fig:Ar2_Equilibria_Distance}, confirm that the proposed rational scheme exactly preserves the equilibrium $\bm{x}^*$, yielding zero numerical distance throughout the simulation, in agreement with Theorem~\ref{thm:Internal_Equilibria}. In contrast, all other methods exhibit a progressive departure from $\bm{x}^*$, indicating their inability to retain the equilibrium property. This loss is also evident from the qualitative behavior of the trajectories in Figure~\ref{fig:Ar2_Comparison}, where the solutions generated by the standard methods exhibit unphysical deviations from the expected dynamics, including negative components and violation of the simplex constraint.

\begin{figure}[!ht]
    \begin{minipage}{73mm}
        \includegraphics[width=\textwidth]{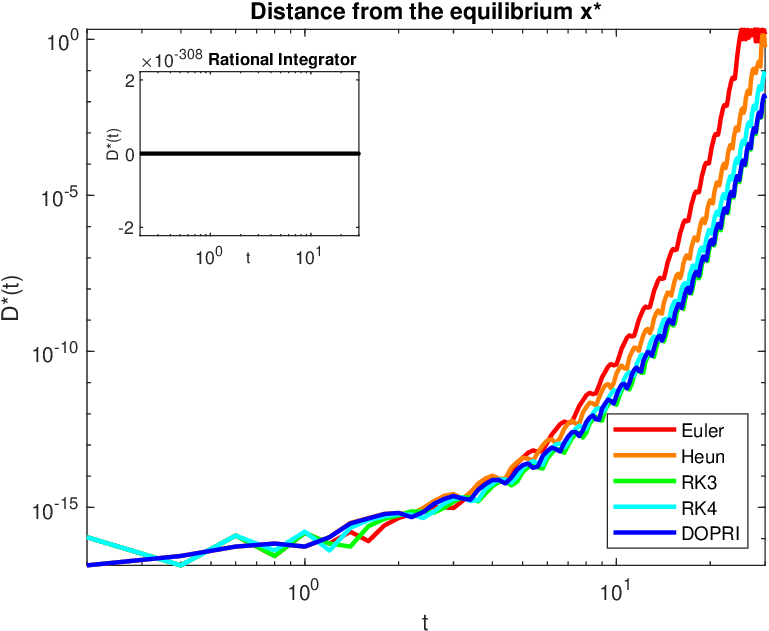}
        \captionsetup{labelformat=empty}
        \caption{}\vspace{-1.5\baselineskip}
        \label{fig:Ar2_Equilibria_Distance}
    \end{minipage}
    \marginnote{\textbf{Figure~\ref{fig:Ar2_Equilibria_Distance}.} \\
   Numerical distance from the equilibrium point computed for \hyperref[test2]{Test~2} by various explicit methods.}[-4.5\baselineskip]
\end{figure}
\begin{figure}[!ht]
    \hspace{0.28cm}\begin{minipage}{\textwidth}
        \includegraphics[width=0.8\textwidth]{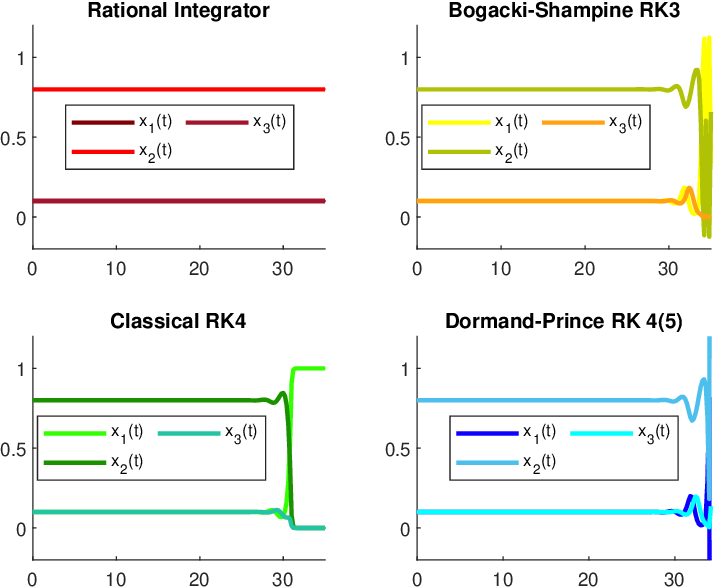}
        \captionsetup{labelformat=empty}
        \caption{}\vspace{-1.5\baselineskip}
        \label{fig:Ar2_Comparison}
    \end{minipage}
    \marginnote{\textbf{Figure~\ref{fig:Ar2_Comparison}.} \\
   Time evolution of the replicator dynamics of \hyperref[test2]{Test~2} under different time discretization schemes starting from the equilibrium $\bm{x}^*\in\Delta^2$}[-16.5\baselineskip]
\end{figure}

\paragraph{Test 3. Discrete quotient rule -}\label{test3}\!\!\!\! To address the challenge of accurately capturing the dynamics of the components’ relative growth, we consider the model \eqref{eq:replicator} with $N=6$ and define
\begin{equation}\label{eq:Test 3}
    f_i(\bm{x}) =  \begin{cases} 
                            4 \sin(6 \pi x_i) +\cos\left(5 \pi x_{1}\right), \qquad \quad &i\in\{1,6\}, \\
                            4 \sin(6 \pi x_i) +\cos\left(5 \pi x_{i+1}\right),  &i=2,\dots,5,
                    \end{cases}
\end{equation}
together with the initial condition 
\begin{equation*}
    \bm{x}^0=\left[\dfrac{2774471}{231910617},\, \dfrac{98983369}{1352447194}, \, \dfrac{97928969}{287827985}, \, \dfrac{12404831}{894275012}, \, \dfrac{14625008}{87442077}, \, \dfrac{38055861}{96713983}\right]^\mathsf{T},
\end{equation*}
deliberately chosen to induce a rich and heterogeneous dynamical behavior over the time interval $[0,5].$  Figure~\ref{fig:Ar3_Simulation} provides insight into the numerical solution obtained by applying the rational integrator \eqref{eq:Rational_Integrator}, with an adaptive stepsize sequence $\{h_n\}$ determined according to the procedure described in Section~\ref{Subsec:VariableH}. The integration is carried out using component-wise absolute and relative tolerances set to $\tau_i^{\text{abs}} = \tau_i^{\text{rel}} = 10^{-10}$ for all $i = 1,\dots,6$ and with an initial steplength $h_0 = 10^{-3}$.

\begin{figure}[!ht]
    \hspace{0.28cm}\begin{minipage}{73mm}
        \includegraphics[width=\textwidth]{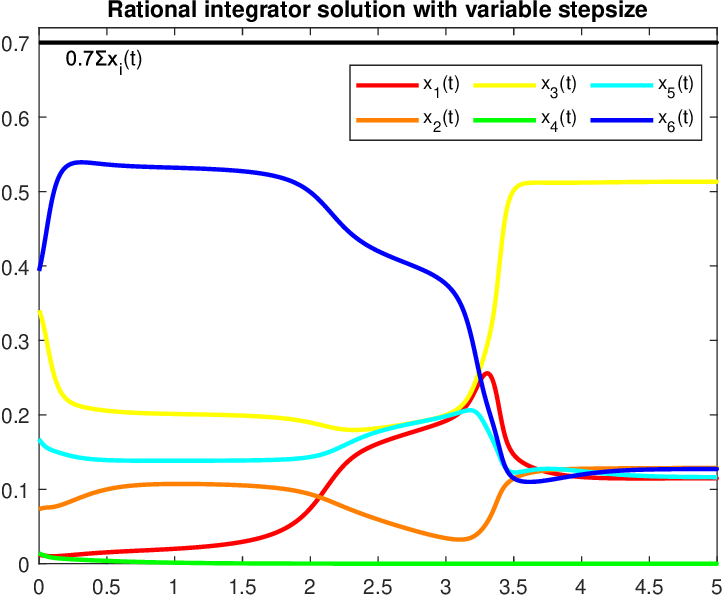}
        \captionsetup{labelformat=empty}
        \caption{}\vspace{-1.5\baselineskip}
        \label{fig:Ar3_Simulation}
    \end{minipage}
    \marginnote{\textbf{Figure~\ref{fig:Ar3_Simulation}.} \\
   Numerical solution computed for \hyperref[test3]{Test~3} by the rational integrator \eqref{eq:Rational_Integrator}. Here, the stepsize is automatically adjusted as described in Section \ref{Subsec:VariableH} }[-4.5\baselineskip]
\end{figure}

Here, we investigate the numerical preservation of the \emph{quotient rule}~\eqref{eq:quotient_rule}. Recalling its discrete counterpart~\eqref{eq:quotient_rule_Discrete}, we define at each time step $n = 0, \ldots, M-1$ the residual matrix $Q^n \in \mathbb{R}^{N \times N}$ as
\begin{equation*}
    Q_{i,j}^n = \frac{x_i^{n+1}}{x_j^{n+1}} - \frac{x_i^n}{x_j^n} \exp\bigl(h_n(f_i(\bm{x}^n) - f_j(\bm{x}^n))\bigr), \qquad \quad i,j=1,\ldots,N,
\end{equation*}
where $\bm{x}^n$ is the approximation of $\bm{x}(t_n)$ computed with stepsize $h_n.$ Furthermore, we introduce the instantaneous residual average
\begin{equation*}
    R^n = \dfrac{1}{N^2}\sum_{i=1}^N \sum_{j=1}^N \left| Q_{i,j}^n \right| ,
\end{equation*}
as a time dependent indicator of the scheme’s ability to reproduce the overall relative growth dynamics of the components. Figure~\ref{fig:Ar3_Residual_time} reports the results obtained using the proposed rational integrator as well as those computed by the Verner Runge--Kutta 8(7) scheme~\cite{Verner_2009}, the Dormand--Prince Runge--Kutta 4(5) method~\cite{dormand1980} and the variable-order backward differentiation formula~\cite{ode15s} implemented in the MATLAB routines \texttt{ode78}, \texttt{ode45} and \texttt{ode15s}, respectively. To ensure a consistent comparison, all the methods are employed with the same sequence of stepsizes, generated by applying the PI controller~\eqref{eq:PI_Controller}, with a tolerance of $10^{-10},$ to the scheme \eqref{eq:Rational_Integrator}. The plots of Figure~\ref{fig:Ar3_Residual_time} clearly show that the rational integrator yields lower values of $R^n$ throughout the simulation, which indicates a markedly improved adherence to the \emph{quotient rule}. This behavior highlights not only the numerical accuracy of \eqref{eq:Rational_Integrator}, but also its enhanced capability to preserve the mutual interaction structure among the components of the system \eqref{eq:replicator}.

\begin{figure}[!ht]
    \begin{minipage}{\textwidth}
        \includegraphics[width=0.8\textwidth]{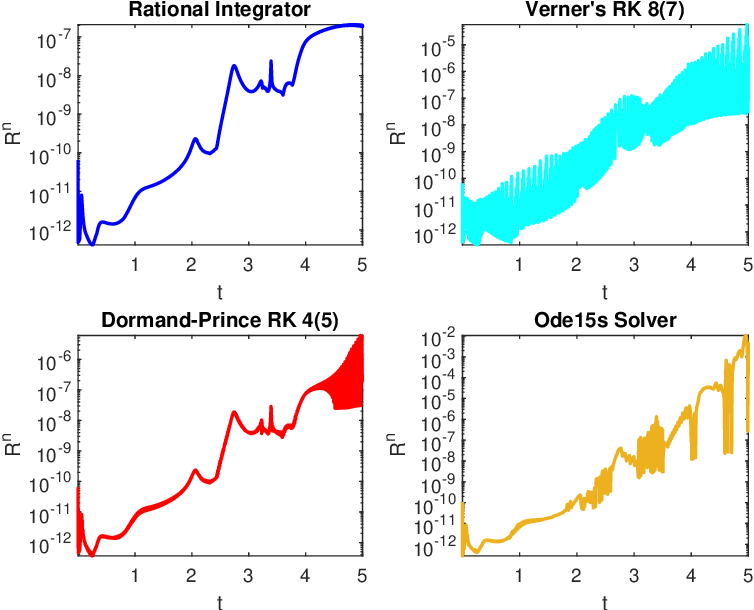}
        \captionsetup{labelformat=empty}
        \caption{}\vspace{-1.5\baselineskip}
        \label{fig:Ar3_Residual_time}
    \end{minipage}
    \marginnote{\textbf{Figure~\ref{fig:Ar3_Residual_time}.} \\
   Instantaneous residual average $R^n$ computed with different variable-stepsize solvers, all using absolute and relative tolerances set to $10^{-10}$. }[-16.5\baselineskip]
\end{figure}

With the aim of numerically validating Theorem~\ref{thm:Discrete_Quotient_Rule}, we consider the fixed-stepsize version of~\eqref{eq:Rational_Integrator} and compute both the global-in-time residual average  $R_Q(h)$ and its experimental decay rate $\varrho$, defined as follows
\begin{equation}\label{eq:residual average}
    R_Q(h) = h\sum_{n=0}^{M} R^n, \qquad \qquad \varrho = \log_2\left( \frac{R(h)}{R\left(h/2\right)}\right).
\end{equation}
The results, reported in Table~\ref{tab:Ar3_R_Order_Test} and Figure~\ref{fig:Ar3_DQR_Convergence}, exhibit a quadratic decay of the residual \eqref{eq:residual average}, in agreement with the second-order accuracy of the rational scheme~\eqref{eq:Rational_Integrator} in approximating the \emph{quotient rule}~\eqref{eq:quotient_rule}.

\begin{figure}[!ht]
    \hspace{0.28cm}\begin{minipage}{73mm}
    \includegraphics[width=\textwidth]{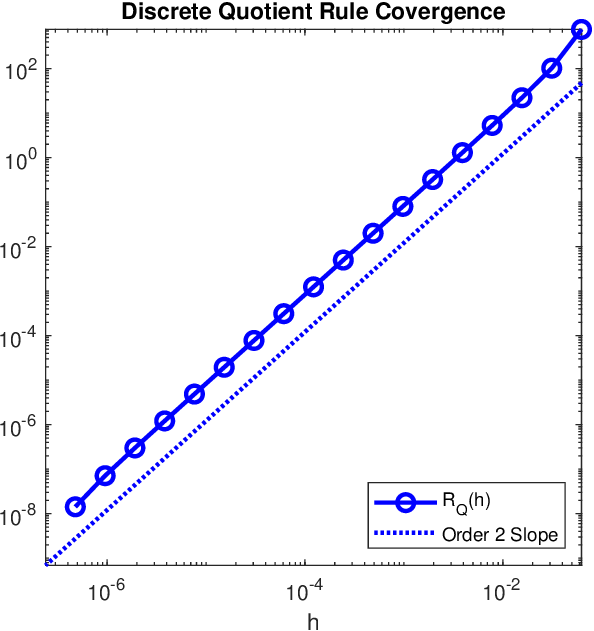}
    \captionsetup{labelformat=empty}
    \caption{}\vspace{-1.5\baselineskip}
    \label{fig:Ar3_DQR_Convergence}
    \end{minipage}
    \marginnote{\textbf{Figure \ref{fig:Ar3_DQR_Convergence}} \\ Experimental convergence of the discrete \emph{quotient rule} residual for the normalized rational integrator~\eqref{eq:Rational_Integrator} applied to \hyperref[test3]{Test~3}.}[-5.5\baselineskip]
\end{figure}
\begin{table}[!ht]
    \begin{minipage}{\linewidth}
      \hspace{0.55cm} 
      \begin{tabular}{|c| c c|c| c c|}
        \multicolumn{6}{c}{\bf Rational Integrator} \\
        \hline
        $h$ & $R_Q(h)$ & $\varrho$ & $h$ & $R_Q(h)$ & $\varrho$ \\
        \hline
        $2^{-4}$  & $7.59 \cdot 10^{2\phantom{-}}$   & -- & $2^{-12}$ & $5.00 \cdot 10^{-3}$ & $2.00$ \\
        $2^{-5}$  & $1.02 \cdot 10^{2\phantom{-}}$   & $2.21$ & $2^{-13}$ & $1.25 \cdot 10^{-3}$ & $2.00$ \\
        $2^{-6}$  & $2.22 \cdot 10^{1\phantom{-}}$   & $2.07$ & $2^{-14}$ & $3.12 \cdot 10^{-4}$ & $2.00$ \\
        $2^{-7}$  & $5.29 \cdot 10^{0\phantom{-}}$   & $2.03$ & $2^{-15}$ & $7.81 \cdot 10^{-5}$ & $2.00$ \\
        $2^{-8}$  & $1.30 \cdot 10^{0\phantom{-}}$   & $2.01$ & $2^{-16}$ & $1.95 \cdot 10^{-5}$ & $2.00$ \\
        $2^{-9}$  & $3.22 \cdot 10^{-1}$  & $2.01$ & $2^{-17}$ & $4.87 \cdot 10^{-6}$ & $2.00$ \\
        $2^{-10}$ & $8.02 \cdot 10^{-2}$  & $2.00$ & $2^{-18}$ & $1.21 \cdot 10^{-6}$ & $2.02$ \\
        $2^{-11}$ & $2.00 \cdot 10^{-2}$  & $2.00$ & $2^{-19}$ & $3.00 \cdot 10^{-7}$ & $2.07$ \\
        \hline
    \end{tabular}
        \captionsetup{labelformat=empty}
        \caption{}\vspace{-1.5\baselineskip}
        \label{tab:Ar3_R_Order_Test}
    \end{minipage}
    \marginnote{\textbf{Table \ref{tab:Ar3_R_Order_Test}.} \\
     Discrete quotient residual and experimental decay rate for the normalized rational scheme \eqref{eq:Rational_Integrator} on \hyperref[test3]{Test~3}.}[-8.5\baselineskip]
\end{table}

\paragraph{Test 4. Adaptive stepsize selection -}\label{test4}\!\!\!\!Our final example addresses the efficient approximation of the solution to \eqref{eq:replicator}, with $\bm{x} \in \Delta^4 \subset \mathbb{R}^5$,
\begin{equation}\label{eq:test5_details}
    \bm{f}^\theta \! (\bm{x}(t)) \!  = \! \theta\begin{bmatrix}
        4 \, (x_2(t) - x_5(t)) + 8 \tanh(3 \, x_4(t) - 9/10) \\
        4 \, (x_3(t) - x_1(t)) + 6 \tanh(2 \, x_5(t) - 2/5) \phantom{.} \\
        4 \, (x_4(t) - x_2(t)) + 7 \tanh(4 \, x_1(t) - 1) \phantom{/5.} \\
        4 \, (x_5(t) - x_3(t)) + 6 \tanh(3 \, x_2(t) - 6/5) \phantom{.} \\
        4 \, (x_1(t) - x_4(t)) + 5 \tanh(5 \, x_3(t) - 7/4) \phantom{.}
    \end{bmatrix},  \quad \bm{x}^0 \!  = \!  \begin{bmatrix}
        0.22 \\ 0.19 \\  0.21 \\ 0.18 \\ 0.20
    \end{bmatrix},
\end{equation}
where $\theta \geq 0$ and $t \in [0,10]$. We begin by integrating problem \eqref{eq:test5_details} with $\theta = 1$, using the rational method \eqref{eq:Rational_Integrator} combined with adaptive stepsize control based on the PI controller \eqref{eq:PI_Controller}. We test different tolerance values, $\tau_i^{\text{abs}} = \tau_i^{\text{rel}} =\tau \in \{10^{-4}, \, 10^{-7}, \, 10^{-10}\}$, and evaluate the absolute simulation errors with respect to a reference solution computed by the Dormand--Prince Runge--Kutta 4(5) scheme~\cite{dormand1980}, as implemented in the MATLAB routine \texttt{ode45} with $\texttt{AbsTol}=\texttt{RelTol}=10^{-13}$. Figure~\ref{fig:Ar4_Errors_Steps} reports, for each tolerance level, the sequence of timesteps selected by the PI controller together with the corresponding $\ell_2$ errors, which remain consistently below the prescribed threshold. The stepsize sequences exhibit a regular and persistent oscillatory pattern, qualitatively similar across all tested values of $\tau$. This response is induced by the structure of the continuous solution to \eqref{eq:replicator}, where some components evolve in a nearly periodic regime. The controller detects and follows these regular oscillations by adjusting the timestep locally, while preserving the overall shape of the sequence.

\begin{figure}[!ht]
    \begin{minipage}{\textwidth}
        \includegraphics[width=0.8\textwidth]{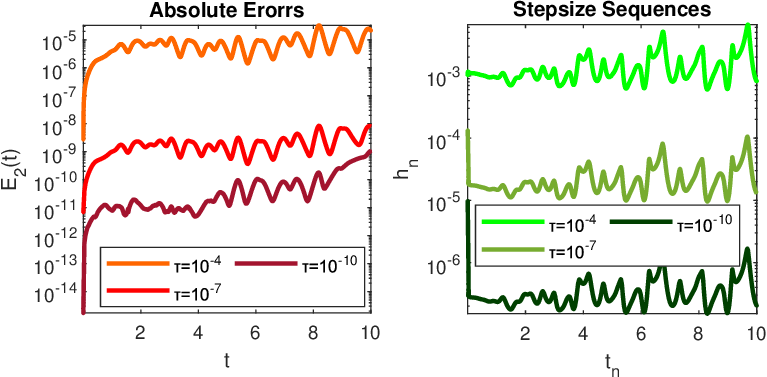}
        \captionsetup{labelformat=empty}
        \caption{}\vspace{-1.5\baselineskip}
        \label{fig:Ar4_Errors_Steps}
    \end{minipage}
    \marginnote{\textbf{Figure~\ref{fig:Ar4_Errors_Steps}.} \\
   Absolute approximation errors and stepsize sequences for the rational integrator \eqref{eq:Rational_Integrator} applied to \hyperref[test4]{Test~4} under different tolerance requirements. Here, $\theta=1.$}[-9\baselineskip]
\end{figure}

The continuous solution structure discussed above makes the efficient integration of the system \eqref{eq:test5_details} particularly challenging, especially for larger values of $\theta$, which induce higher-frequency oscillations. In such cases, even high-order variable-step solvers may fail to capture the underlying dynamics and yield qualitatively incorrect trajectories. As an illustrative example, we reconsider problem \eqref{eq:test5_details} with $\theta = 5$, and compare the numerical solutions produced by the rational integrator with those obtained using MATLAB's \texttt{ode78}, \texttt{ode45} and \texttt{ode15s} solvers, all configured with tolerances set to $10^{-10}$. The simulations obtained with the different methods are reported in Figure~\ref{fig:Ar4_Solutions_Comparison}. As shown therein, the rational integrator, by virtue of its structure-preserving formulation, is the only scheme capable of reliably adapting to the periodic changes of the underlying dynamics. As a matter of fact, the solutions computed with the other solvers progressively deteriorate due to the accumulation of structure-inconsistent effects. In particular, the trajectories produced by \texttt{ode78} and \texttt{ode45} exhibit numerical dissipation, with some components becoming negative or vanishing, eventually leading to a collapse of the total mass $\bm{e}^\mathsf{T} \bm{x}^n$. Conversely, the solution computed by \texttt{ode15s} suffers from spurious numerical artifacts that result in a nonphysical increase in the total mass that compromises the reliability of the simulation.

\begin{figure}[!ht]
    \begin{minipage}{\textwidth}
        \includegraphics[width=0.8\textwidth]{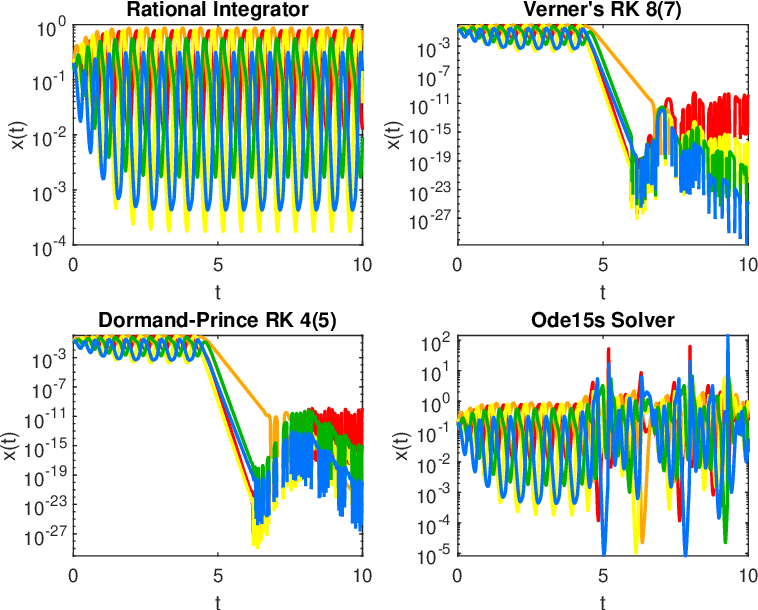}
        \captionsetup{labelformat=empty}
        \caption{}\vspace{-1.5\baselineskip}
        \label{fig:Ar4_Solutions_Comparison}
    \end{minipage}
    \marginnote{\textbf{Figure~\ref{fig:Ar4_Solutions_Comparison}.} \\
   Numerical solutions of \hyperref[test4]{Test~4}, with $\theta=5,$ computed using different variable-stepsize solvers with absolute and relative tolerances set to $10^{-10}$. }[-16.5\baselineskip]
\end{figure}

\section{Conclusions}\label{sec:Conclusions}

In this work, we developed a structure-preserving rational integrator designed for the numerical approximation of the replicator equation on the probability simplex. The method combines a two-stage rational discretization with a normalization step, yielding an explicit scheme that is straightforward to implement and unconditionally preserves fundamental geometric and dynamical properties of the continuous model. In particular, the integrator guarantees non-negativity, mass conservation, invariance of all simplex faces and exact preservation of both boundary and internal equilibria. We proved second-order convergence and constructed an embedded first-order approximation to enable reliable local error estimation and an efficient adaptive time-stepping strategy based on PI control. Additionally, the adherence to the \emph{quotient rule} governing the reciprocal evolution of component ratios is demonstrated. Extensive numerical tests confirm the theoretical results and highlight the method’s superior performance in preserving qualitative features of the replicator dynamics, even in challenging regimes involving periodic and oscillatory solutions. Comparisons with standard variable-step solvers highlight the advantages of the proposed structure-preserving approach.

The present work opens several directions for further investigation. First, the extension of the proposed structure-preserving integrator to replicator equation variants incorporating mutation effects, namely replicator-mutator systems, which are widely studied in evolutionary dynamics \cite{Amadori_Mutations_1, Amadori_Mutations_2}. Another natural extension concerns non-local and integro-differential replicator models, as analyzed in \cite{AMADORI_BRIANI_NATALINI_2016,Alfaro_2023,Alfaro_2020} and related literature. Additional research topics include delayed replicator dynamics, stochastic formulations and heterogeneous multi-population systems, where preserving structural properties remains a key challenge. Finally, exploring the application of rational approximation based integrators to other differential models with inherent structural properties, such as production-destruction \cite{PDS,OFFNER202015,MPSL,GeCO_Izgin} systems, is a promising area for further study. 
\medskip

\paragraph{Data availability statement} The codes implementing the numerical methods discussed in this work are available from the author upon reasonable request. No new data were generated or analyzed in this study and references to existing data are provided in the manuscript.

\paragraph{Acknowledgments} This work has been performed under the Project PE 0000020 CHANGES - CUP\_B53C22003890006, NRP Mission 4 Component 2 Investment 1.3, Funded by the European Union - NextGenerationEU and in the auspices of the Italian National Group for Scientific Computing (GNCS) of the National Institute for Advanced Mathematics (INdAM).


\bibliography{References}

\begin{thebibliography}{10}

\bibitem{Dynamic}
H.~Al-Kahby, F.~Dannan, and S.~Elaydi.
\newblock Non-standard discretization methods for some biological models.
\newblock In {\em Applications of nonstandard finite difference schemes ({A}tlanta, {GA}, 1999)}, pages 155--180. World Sci. Publ., River Edge, NJ, 2000.
\newblock \href {https://doi.org/10.1142/9789812813251\_0004} {\path{doi:10.1142/9789812813251\_0004}}.

\bibitem{Alfaro_2023}
M.~Alfaro, P.~Gabriel, and O.~Kavian.
\newblock Confining integro-differential equations originating from evolutionary biology: {Ground} states and long time dynamics.
\newblock {\em Discrete and Continuous Dynamical Systems - B}, 28(12):5905--5933, 2023.
\newblock \href {https://doi.org/10.3934/dcdsb.2022120} {\path{doi:10.3934/dcdsb.2022120}}.

\bibitem{Alfaro_2020}
M.~Alfaro and M.~Veruete.
\newblock Density dependent replicator-mutator models in directed evolution.
\newblock {\em Discrete and Continuous Dynamical Systems - B}, 25(6):2203--2221, 2020.
\newblock \href {https://doi.org/10.3934/dcdsb.2019224} {\path{doi:10.3934/dcdsb.2019224}}.

\bibitem{Natalini2}
A.~L. Amadori, A.~Boccabella, and R.~Natalini.
\newblock A hyperbolic model of spatial evolutionary game theory.
\newblock {\em Communications on Pure and Applied Analysis}, 11(3):981--1002, 2012.
\newblock \href {https://doi.org/10.3934/cpaa.2012.11.981} {\path{doi:10.3934/cpaa.2012.11.981}}.

\bibitem{AMADORI_BRIANI_NATALINI_2016}
A.~L. Amadori, M.~Briani, and R.~Natalini.
\newblock {A non-local rare mutations model for quasispecies and prisoner’s dilemma: Numerical assessment of qualitative behaviour}.
\newblock {\em European Journal of Applied Mathematics}, 27(1):87–110, 2016.
\newblock \href {https://doi.org/10.1017/S0956792515000352} {\path{doi:10.1017/S0956792515000352}}.

\bibitem{Amadori_Mutations_1}
A.~L. Amadori, A.~Calzolari, R.~Natalini, and B.~Torti.
\newblock Rare mutations in evolutionary dynamics.
\newblock {\em Journal of Differential Equations}, 259(11):6191--6214, 2015.
\newblock \href {https://doi.org/10.1016/j.jde.2015.07.021} {\path{doi:10.1016/j.jde.2015.07.021}}.

\bibitem{Amadori_Mutations_2}
A.~L. Amadori, R.~Natalini, and D.~Palmigiani.
\newblock A rare mutation model in a spatial heterogeneous environment.
\newblock {\em Ecological Complexity}, 34:188--197, 2018.
\newblock \href {https://doi.org/10.1016/j.ecocom.2017.10.003} {\path{doi:10.1016/j.ecocom.2017.10.003}}.

\bibitem{Lubuma_Chem}
R.~Anguelov, Y.~Dumont, and J.~M.-S. Lubuma.
\newblock {On nonstandard finite difference schemes in biosciences}.
\newblock {\em AIP Conference Proceedings}, 1487(1):212--223, 10 2012.
\newblock \href {https://doi.org/10.1063/1.4758961} {\path{doi:10.1063/1.4758961}}.

\bibitem{Anguelov}
R.~Anguelov and J.M.-S. Lubuma.
\newblock Contributions to the mathematics of the nonstandard finite difference method and applications.
\newblock {\em Numerical Methods for Partial Differential Equations}, 17(5):518 – 543, 2001.
\newblock \href {https://doi.org/10.1002/num.1025} {\path{doi:10.1002/num.1025}}.

\bibitem{Pade_Book}
G.~A. Baker and P.~Graves-Morris.
\newblock {\em {Padé Approximants Second Edition}}.
\newblock Cambridge University Press, 1996.
\newblock \href {https://doi.org/10.1017/cbo9780511530074} {\path{doi:10.1017/cbo9780511530074}}.

\bibitem{Natalini1}
A.~Boccabella, R.~Natalini, and L.~Pareschi.
\newblock On a continuous mixed strategies model for evolutionary game theory.
\newblock {\em Kinetic and Related Models}, 4(1):187--213, 2011.
\newblock \href {https://doi.org/10.3934/krm.2011.4.187} {\path{doi:10.3934/krm.2011.4.187}}.

\bibitem{bogacki1989}
P.~Bogacki and L.F. Shampine.
\newblock {A 3(2) pair of Runge - Kutta formulas}.
\newblock {\em Applied Mathematics Letters}, 2(4):321--325, 1989.
\newblock \href {https://doi.org/10.1016/0893-9659(89)90079-7} {\path{doi:10.1016/0893-9659(89)90079-7}}.

\bibitem{Bomze_1997}
I.~M. Bomze.
\newblock {Evolution towards the Maximum Clique}.
\newblock {\em Journal of Global Optimization}, 10(2):143–164, March 1997.
\newblock \href {https://doi.org/10.1023/a:1008230200610} {\path{doi:10.1023/a:1008230200610}}.

\bibitem{boyd2004convex}
S.~Boyd and L.~Vandenberghe.
\newblock {\em Convex Optimization}.
\newblock Cambridge University Press, 2004.

\bibitem{Butcher_2016}
J.~C. Butcher.
\newblock {\em {Numerical Methods for Ordinary Differential Equations}}.
\newblock Wiley, 2016.
\newblock \href {https://doi.org/10.1002/9781119121534} {\path{doi:10.1002/9781119121534}}.

\bibitem{MPSL}
S.~Cacace, A.~Oliviero, and M.~Pezzella.
\newblock {Modified Patankar Semi-Lagrangian Scheme for the Optimal Control of Production-Destruction systems}, 2025.
\newblock \href {https://arxiv.org/abs/2501.13085} {\path{arXiv:2501.13085}}, \href {https://doi.org/10.48550/arXiv.2501.13085} {\path{doi:10.48550/arXiv.2501.13085}}.

\bibitem{Cardone_Fract}
A.~Cardone, G.~Frasca-Caccia, and B.~Paternoster.
\newblock Non-standard schemes for time-fractional reaction–advection–diffusion problems.
\newblock {\em Journal of Computational and Applied Mathematics}, 471:116757, 2026.
\newblock \href {https://doi.org/10.1016/j.cam.2025.116757} {\path{doi:10.1016/j.cam.2025.116757}}.

\bibitem{Celledoni2014}
E.~Celledoni, H.~Marthinsen, and B.~Owren.
\newblock An introduction to lie group integrators – basics, new developments and applications.
\newblock {\em Journal of Computational Physics}, 257:1040--1061, 2014.
\newblock Physics-compatible numerical methods.
\newblock \href {https://doi.org/10.1016/j.jcp.2012.12.031} {\path{doi:10.1016/j.jcp.2012.12.031}}.

\bibitem{CdS}
M.~Ceseri, R.~Natalini, and M.~Pezzella.
\newblock {An Integro-differential Model of Cadmium Yellow Photodegradation}, 2024.
\newblock \href {https://arxiv.org/abs/2411.06997} {\path{arXiv:2411.06997}}, \href {https://doi.org/10.48550/arXiv.2411.06997} {\path{doi:10.48550/arXiv.2411.06997}}.

\bibitem{RPS_Closed_2}
S.~Chakraborty, I.~Agarwal, and S.~Chakraborty.
\newblock {Replicator-mutator dynamics of the rock-paper-scissors game: Learning through mistakes}.
\newblock {\em Phys. Rev. E}, 109:034404, Mar 2024.
\newblock \href {https://doi.org/10.1103/PhysRevE.109.034404} {\path{doi:10.1103/PhysRevE.109.034404}}.

\bibitem{Conte2022}
D.~Conte, N.~Guarino, G.~Pagano, and B.~Paternoster.
\newblock On the advantages of nonstandard finite difference discretizations for differential problems.
\newblock {\em Numerical Analysis and Applications}, 15(3):219--235, Sep 2022.
\newblock \href {https://doi.org/10.1134/S1995423922030041} {\path{doi:10.1134/S1995423922030041}}.

\bibitem{Dang}
Q.~A. Dang and M.~T. Hoang.
\newblock Dynamically consistent discrete metapopulation model.
\newblock {\em Journal of Difference Equations and Applications}, 22(9):1325--1349, 2016.
\newblock \href {https://doi.org/10.1080/10236198.2016.1197213} {\path{doi:10.1080/10236198.2016.1197213}}.

\bibitem{Clustering}
M.~Donoser.
\newblock {Replicator Graph Clustering}.
\newblock In {\em Proceedings of the British Machine Vision Conference}. BMVA Press, 2013.
\newblock \href {https://doi.org/10.5244/C.27.38} {\path{doi:10.5244/C.27.38}}.

\bibitem{dormand1980}
J.R. Dormand and P.J. Prince.
\newblock {A family of embedded Runge-Kutta formulae}.
\newblock {\em Journal of Computational and Applied Mathematics}, 6(1):19--26, 1980.
\newblock \href {https://doi.org/10.1016/0771-050X(80)90013-3} {\path{doi:10.1016/0771-050X(80)90013-3}}.

\bibitem{Ewens_2004}
W.~J. Ewens.
\newblock {\em {Mathematical Population Genetics}}.
\newblock Springer New York, 2004.
\newblock \href {https://doi.org/10.1007/978-0-387-21822-9} {\path{doi:10.1007/978-0-387-21822-9}}.

\bibitem{Gjini_2023}
E.~Gjini and S.~Madec.
\newblock Towards a mathematical understanding of invasion resistance in multispecies communities.
\newblock {\em Royal Society Open Science}, 10(11), 2023.
\newblock \href {https://doi.org/10.1098/rsos.231034} {\path{doi:10.1098/rsos.231034}}.

\bibitem{Gupta}
M.~Gupta, J.~M. Slezak, F.~Alalhareth, S.~Roy, and H.~V. Kojouharov.
\newblock Second-order nonstandard explicit {Euler} method.
\newblock {\em AIP Conference Proceedings}, 2302(1):110003, 12 2020.
\newblock \href {https://doi.org/10.1063/5.0033534} {\path{doi:10.1063/5.0033534}}.

\bibitem{Gustafsson_1991}
K.~Gustafsson.
\newblock {Control theoretic techniques for stepsize selection in explicit Runge-Kutta methods}.
\newblock {\em ACM Transactions on Mathematical Software}, 17(4):533–554, 1991.
\newblock \href {https://doi.org/10.1145/210232.210242} {\path{doi:10.1145/210232.210242}}.

\bibitem{Hairer2006}
E.~Hairer, C.~Lubich, and G.~Wanner.
\newblock {\em Geometric Numerical Integration: Structure-Preserving Algorithms for Ordinary Differential Equations}.
\newblock Springer-Verlag, Berlin, Heidelberg, 2 edition, 2006.
\newblock \href {https://doi.org/10.1007/3-540-30666-8} {\path{doi:10.1007/3-540-30666-8}}.

\bibitem{Hairer}
E.~Hairer, S.~Norsett, and G.~Wanner.
\newblock {\em {Solving Ordinary Differential Equations I: Nonstiff Problems}}, volume~8.
\newblock Springer Series in Computational Mathematics, 1993.
\newblock \href {https://doi.org/10.1007/978-3-540-78862-1} {\path{doi:10.1007/978-3-540-78862-1}}.

\bibitem{Hairer_II}
E.~Hairer and G.~Wanner.
\newblock {\em {Solving Ordinary Differential Equations II}}.
\newblock Springer Berlin Heidelberg, 1996.
\newblock \href {https://doi.org/10.1007/978-3-642-05221-7} {\path{doi:10.1007/978-3-642-05221-7}}.

\bibitem{heun}
K.~Heun.
\newblock {Neue Methode zur Approximativen Integration der Differentialgleichungen einer Unabhängigen Veränderlichen}.
\newblock {\em Zeitschrift für Mathematik und Physik}, 45:23--38, 1900.

\bibitem{Hoang03102023}
M.~T. Hoang.
\newblock A novel second-order nonstandard finite difference method preserving dynamical properties of a general single-species model.
\newblock {\em International Journal of Computer Mathematics}, 100(10):2047--2062, 2023.
\newblock \href {https://doi.org/10.1080/00207160.2023.2248304} {\path{doi:10.1080/00207160.2023.2248304}}.

\bibitem{NSFD_High}
M.~T. Hoang.
\newblock High-order nonstandard finite difference methods preserving dynamical properties of one-dimensional dynamical systems.
\newblock {\em Numerical Algorithms}, Mar 2024.
\newblock \href {https://doi.org/10.1007/s11075-024-01792-1} {\path{doi:10.1007/s11075-024-01792-1}}.

\bibitem{Hoang_Second}
M.~T. Hoang and M.~Ehrhardt.
\newblock {A second-order nonstandard finite difference method for a general Rosenzweig–MacArthur predator–prey model}.
\newblock {\em Journal of Computational and Applied Mathematics}, 444:115752, 2024.
\newblock \href {https://doi.org/10.1016/j.cam.2024.115752} {\path{doi:10.1016/j.cam.2024.115752}}.

\bibitem{Random}
J.~Hofbauer and W.~H. Sandholm.
\newblock Evolution in games with randomly disturbed payoffs.
\newblock {\em Journal of Economic Theory}, 132(1):47--69, 2007.
\newblock \href {https://doi.org/10.1016/j.jet.2005.05.011} {\path{doi:10.1016/j.jet.2005.05.011}}.

\bibitem{Hofbauer1998}
J.~Hofbauer and K.~Sigmund.
\newblock {\em {Evolutionary Games and Population Dynamics}}.
\newblock Cambridge University Press, 1998.
\newblock \href {https://doi.org/10.1017/cbo9781139173179} {\path{doi:10.1017/cbo9781139173179}}.

\bibitem{Hofbauer_2003}
J.~Hofbauer and K.~Sigmund.
\newblock Evolutionary game dynamics.
\newblock {\em Bulletin of the American Mathematical Society}, 40(4):479–519, July 2003.
\newblock \href {https://doi.org/10.1090/s0273-0979-03-00988-1} {\path{doi:10.1090/s0273-0979-03-00988-1}}.

\bibitem{GeCO_Izgin}
T.~Izgin, S.~Kopecz, A.~Martiradonna, and A.~Meister.
\newblock {On the dynamics of first and second order GeCo and gBBKS schemes}.
\newblock {\em Applied Numerical Mathematics}, 193:43--66, 2023.
\newblock \href {https://doi.org/10.1016/j.apnum.2023.07.014} {\path{doi:10.1016/j.apnum.2023.07.014}}.

\bibitem{PDS}
G.~Izzo, E.~Messina, M.~Pezzella, and A.~Vecchio.
\newblock Modified patankar linear multistep methods for production-destruction systems.
\newblock {\em Journal of Scientific Computing}, 102(3), February 2025.
\newblock \href {https://doi.org/10.1007/s10915-025-02804-5} {\path{doi:10.1007/s10915-025-02804-5}}.

\bibitem{Kojouharov}
H.~V. Kojouharov, S.~Roy, M.~Gupta, F.~Alalhareth, and J.~M. Slezak.
\newblock A second-order modified nonstandard theta method for one-dimensional autonomous differential equations.
\newblock {\em Applied Mathematics Letters}, 112:106775, 2021.
\newblock \href {https://doi.org/10.1016/j.aml.2020.106775} {\path{doi:10.1016/j.aml.2020.106775}}.

\bibitem{RPS_Closed}
R.~Lahkar and W.H. Sandholm.
\newblock The projection dynamic and the geometry of population games.
\newblock {\em Games and Economic Behavior}, 64(2):565--590, 2008.
\newblock Special Issue in Honor of Michael B. Maschler.
\newblock \href {https://doi.org/10.1016/j.geb.2008.02.002} {\path{doi:10.1016/j.geb.2008.02.002}}.

\bibitem{LEWIS2003141}
D.~Lewis and N.~Nigam.
\newblock Geometric integration on spheres and some interesting applications.
\newblock {\em Journal of Computational and Applied Mathematics}, 151(1):141--170, 2003.
\newblock \href {https://doi.org/10.1016/S0377-0427(02)00743-4} {\path{doi:10.1016/S0377-0427(02)00743-4}}.

\bibitem{Li_2025}
J.~Li, X.~Wang, C.~Li, and B.~Zhang.
\newblock Replicator dynamics on heterogeneous networks.
\newblock {\em Journal of Mathematical Biology}, 90(2), 2025.
\newblock \href {https://doi.org/10.1007/s00285-024-02177-7} {\path{doi:10.1007/s00285-024-02177-7}}.

\bibitem{Liu}
C.~Liu, C.~Wang, and Y.~Wang.
\newblock A structure-preserving, operator splitting scheme for reaction-diffusion equations with detailed balance.
\newblock {\em Journal of Computational Physics}, 436:110253, 2021.
\newblock \href {https://doi.org/10.1016/j.jcp.2021.110253} {\path{doi:10.1016/j.jcp.2021.110253}}.

\bibitem{Madec_2020}
S.~Madec and E.~Gjini.
\newblock {Predicting N-Strain Coexistence from Co-colonization Interactions: Epidemiology Meets Ecology and the Replicator Equation}.
\newblock {\em Bulletin of Mathematical Biology}, 82(11), 2020.
\newblock \href {https://doi.org/10.1007/s11538-020-00816-w} {\path{doi:10.1007/s11538-020-00816-w}}.

\bibitem{GeC01}
A.~Martiradonna, G.~Colonna, and F.~Diele.
\newblock {GeCo: Geometric Conservative nonstandard schemes for biochemical systems}.
\newblock {\em Applied Numerical Mathematics}, 155:38--57, 2020.
\newblock Structural Dynamical Systems: Computational Aspects held in Monopoli (Italy) on June 12-15, 2018.
\newblock \href {https://doi.org/10.1016/j.apnum.2019.12.004} {\path{doi:10.1016/j.apnum.2019.12.004}}.

\bibitem{MAYNARDSMITH1974209}
J.~{Maynard Smith}.
\newblock The theory of games and the evolution of animal conflicts.
\newblock {\em Journal of Theoretical Biology}, 47(1):209--221, 1974.
\newblock \href {https://doi.org/10.1016/0022-5193(74)90110-6} {\path{doi:10.1016/0022-5193(74)90110-6}}.

\bibitem{MPV_JCD}
E.~Messina, M.~Pezzella, and A.~Vecchio.
\newblock A non-standard numerical scheme for an age-of-infection epidemic model.
\newblock {\em Journal of Computational Dynamics}, 9(2):239--252, 2022.
\newblock \href {https://doi.org/10.3934/jcd.2021029} {\path{doi:10.3934/jcd.2021029}}.

\bibitem{MPV_Axioms}
E.~Messina, M.~Pezzella, and A.~Vecchio.
\newblock Positive numerical approximation of integro-differential epidemic model.
\newblock {\em Axioms}, 11(2), 2022.
\newblock \href {https://doi.org/10.3390/axioms11020069} {\path{doi:10.3390/axioms11020069}}.

\bibitem{MPV_MBE}
E.~Messina, M.~Pezzella, and A.~Vecchio.
\newblock Nonlocal finite difference discretization of a class of renewal equation models for epidemics.
\newblock {\em Mathematical Biosciences and Engineering}, 20(7):11656--11675, 2023.
\newblock \href {https://doi.org/10.3934/mbe.2023518} {\path{doi:10.3934/mbe.2023518}}.

\bibitem{MPV_mixing}
E.~Messina, M.~Pezzella, and A.~Vecchio.
\newblock A long-time behavior preserving numerical scheme for age-of-infection epidemic models with heterogeneous mixing.
\newblock {\em Applied Numerical Mathematics}, 200:344--357, 2024.
\newblock New Trends in Approximation Methods and Numerical Analysis (FAATNA20$>$22).
\newblock \href {https://doi.org/10.1016/j.apnum.2023.04.009} {\path{doi:10.1016/j.apnum.2023.04.009}}.

\bibitem{Mickens_NSFD}
R.~E. Mickens.
\newblock {\em Nonstandard Finite Difference Models of Differential Equations}.
\newblock WORLD SCIENTIFIC, 1993.
\newblock \href {https://doi.org/10.1142/2081} {\path{doi:10.1142/2081}}.

\bibitem{Nowak2006}
M.~A. Nowak.
\newblock {\em {Evolutionary Dynamics: Exploring the Equations of Life}}.
\newblock Harvard University Press, 2006.

\bibitem{Patidar}
K.~C. Patidar.
\newblock Nonstandard finite difference methods: recent trends and further developments.
\newblock {\em Journal of Difference Equations and Applications}, 22(6):817--849, 2016.
\newblock \href {https://doi.org/10.1080/10236198.2016.1144748} {\path{doi:10.1080/10236198.2016.1144748}}.

\bibitem{Pavan2007}
M.~Pelillo and M.~Pavan.
\newblock Dominant sets and pairwise clustering.
\newblock {\em IEEE Transactions on Pattern Analysis and Machine Intelligence}, 29(1):167--172, 2007.
\newblock \href {https://doi.org/10.1109/TPAMI.2007.250608} {\path{doi:10.1109/TPAMI.2007.250608}}.

\bibitem{Pezzella_ESAIM}
{Pezzella, M.}
\newblock High order positivity-preserving numerical methods for a non-local photochemical model.
\newblock {\em ESAIM: M2AN}, 59(3):1763--1790, 2025.
\newblock \href {https://doi.org/10.1051/m2an/2025041} {\path{doi:10.1051/m2an/2025041}}.

\bibitem{Fract}
P.C. Podila, R.~Mishra, and H.~Ramos.
\newblock {A Non-Standard Finite Difference Scheme for Time-Fractional Singularly Perturbed Convection–Diffusion Problems}.
\newblock {\em Fractal and Fractional}, 9(6), 2025.
\newblock \href {https://doi.org/10.3390/fractalfract9060333} {\path{doi:10.3390/fractalfract9060333}}.

\bibitem{Networks2}
R.~J. Requejo and A.~D\'{\i}az-Guilera.
\newblock Replicator dynamics with diffusion on multiplex networks.
\newblock {\em Phys. Rev. E}, 94:022301, Aug 2016.
\newblock \href {https://doi.org/10.1103/PhysRevE.94.022301} {\path{doi:10.1103/PhysRevE.94.022301}}.

\bibitem{Sandholm2010}
W.~H. Sandholm.
\newblock {Population Games And Evolutionary Dynamics}.
\newblock In {\em Economic learning and social evolution}, 2010.

\bibitem{Sato2002}
Y.~Sato and J.~P. Crutchfield.
\newblock Coupled replicator equations for the dynamics of learning in multiagent systems.
\newblock {\em Phys. Rev. E}, 67:015206, Jan 2003.
\newblock \href {https://doi.org/10.1103/PhysRevE.67.015206} {\path{doi:10.1103/PhysRevE.67.015206}}.

\bibitem{SCHUSTER}
P.~Schuster and K.~Sigmund.
\newblock Replicator dynamics.
\newblock {\em Journal of Theoretical Biology}, 100(3):533--538, 1983.
\newblock \href {https://doi.org/10.1016/0022-5193(83)90445-9} {\path{doi:10.1016/0022-5193(83)90445-9}}.

\bibitem{ode15s}
L.~F. Shampine and M.~W. Reichelt.
\newblock {The MATLAB ODE Suite}.
\newblock {\em SIAM Journal on Scientific Computing}, 18(1):1--22, 1997.
\newblock \href {https://doi.org/10.1137/S1064827594276424} {\path{doi:10.1137/S1064827594276424}}.

\bibitem{Sharma}
H.~Sharma, M.~Patil, and C.~Woolsey.
\newblock A review of structure-preserving numerical methods for engineering applications.
\newblock {\em Computer Methods in Applied Mechanics and Engineering}, 366:113067, 2020.
\newblock \href {https://doi.org/10.1016/j.cma.2020.113067} {\path{doi:10.1016/j.cma.2020.113067}}.

\bibitem{Sigmund1986}
K.~Sigmund.
\newblock {\em A Survey of Replicator Equations}, pages 88--104.
\newblock Springer Berlin Heidelberg, Berlin, Heidelberg, 1986.
\newblock \href {https://doi.org/10.1007/978-3-642-70953-1_4} {\path{doi:10.1007/978-3-642-70953-1_4}}.

\bibitem{SMITH_1973}
J.~Maynard Smith and G.~R. Price.
\newblock {The Logic of Animal Conflict}.
\newblock {\em Nature}, 246(5427):15–18, 1973.
\newblock \href {https://doi.org/10.1038/246015a0} {\path{doi:10.1038/246015a0}}.

\bibitem{PI2}
{Söderlind, G. and Gustafsson, K. and Lundh, M.}
\newblock {A PI Stepsize Control for the Numerical Solution of Ordinary Differential Equations}.
\newblock {\em {BIT (Nordisk Tidskrift för Informationsbehandling)}}, {28}({2}):{270--287}, {1988}.

\bibitem{Taylor1978}
P.~D. Taylor and L.~B. Jonker.
\newblock Evolutionarily stable strategies and game dynamics.
\newblock {\em Mathematical Biosciences}, 40(1–2):145--156, 1978.
\newblock \href {https://doi.org/10.1016/0025-5564(78)90077-9} {\path{doi:10.1016/0025-5564(78)90077-9}}.

\bibitem{Varga_2024}
T.~Varga.
\newblock Replicator dynamics generalized for evolutionary matrix games under time constraints.
\newblock {\em Journal of Mathematical Biology}, 90(1), 2024.
\newblock \href {https://doi.org/10.1007/s00285-024-02170-0} {\path{doi:10.1007/s00285-024-02170-0}}.

\bibitem{Verner_2009}
J.~H. Verner.
\newblock {Numerically optimal Runge–Kutta pairs with interpolants}.
\newblock {\em Numerical Algorithms}, 53(2–3):383–396, April 2009.
\newblock \href {https://doi.org/10.1007/s11075-009-9290-3} {\path{doi:10.1007/s11075-009-9290-3}}.

\bibitem{Weibull1995}
J.W. Weibull.
\newblock {\em Evolutionary Game Theory}.
\newblock Mit Press. MIT Press, 1997.

\bibitem{OFFNER202015}
P.~Öffner and D.~Torlo.
\newblock Arbitrary high-order, conservative and positivity preserving {Patankar-type} deferred correction schemes.
\newblock {\em Applied Numerical Mathematics}, 153:15--34, 2020.
\newblock \href {https://doi.org/10.1016/j.apnum.2020.01.025} {\path{doi:10.1016/j.apnum.2020.01.025}}.

\end{thebibliography}

\bibliographystyle{plainurl}

\end{document}